\documentclass{amsart}
\usepackage{amsmath,amsthm,amssymb,amsfonts}

  \usepackage{epsfig}
\usepackage[colorlinks=true]{hyperref}

\hypersetup{urlcolor=blue, citecolor=blue}

\def\doi#1{   {\href{http://dx.doi.org/#1}
   {{\mdseries\ttfamily DOI}}}}

\def\arXiv#1{   {\href{http://arxiv.org/abs/#1}
   {{\mdseries\ttfamily arXiv:#1}}}}

\usepackage{graphics}

\def\e{\varepsilon}

\newcommand{\gm}{\mathfrak{g}} 

\newcommand{\al}{\alpha}    \newcommand{\be}{\beta}
\newcommand{\de}{\delta}    \newcommand{\De}{\Delta}
  \newcommand{\ep}{\varepsilon}
  
    \newcommand{\la}{\lambda}
    
\newcommand{\om}{\omega}    \newcommand{\Om}{\Omega}
    \newcommand{\Ga}{\Gamma}
\newcommand{\R}{\mathbb{R}}

\newcommand{\pt}{\partial_t}\newcommand{\pa}{\partial}
\newcommand{\les}{{\lesssim}}
\newcommand{\beeq}{\begin{equation}}\newcommand{\eneq}{\end{equation}}

\newcommand{\Sp}{{\mathbb S}}\def\CO{\mathcal {O}}

\newcommand{\supp}{\text{supp}}

\newenvironment{prf}{\noindent {\bf Proof.} }{\endprf\par}
\def \endprf{\hfill  {\vrule height6pt width6pt depth0pt}\medskip}

\def\<{\langle}             \def\>{\rangle}
\def\({\left(}                 \def\){\right)}
\numberwithin{equation}{section}
\newtheorem{thm}{Theorem}[section]
 
 \newtheorem{lem}[thm]{Lemma}

 \newtheorem{rem}[thm]{Remark}

\title[Lifespan for semilinear wave equations in $2$-D exterior domains]
{Lifespan estimates for $2$-dimensional semilinear wave equations in asymptotically Euclidean exterior domains}
\author{Ning-An Lai}
\address{Institute of Nonlinear Analysis and Department of Mathematics\\
		Lishui University\\Lishui 323000, P. R. China}
		\email{ninganlai@lsu.edu.cn}

\author{Mengyun Liu$^{*}$}\thanks{* Corresponding author}
\address{Department of Mathematics\\Zhejiang Sci-Tech University\\Hangzhou 310018, P. R. China }
\email{mengyunliu@zstu.edu.cn}

\author{Kyouhei Wakasa}
\address{Department of Creative Engineering, National Institute of Technology, Kushiro College, 2-32-1 Otanoshike-Nishi, Kushiro-Shi, Hokkaido 084-0916, Japan}\email{wakasa@kushiro-ct.ac.jp}

\author{Chengbo Wang}
\address{School of Mathematical Sciences\\ Zhejiang University\\Hangzhou 310027, P. R. China}\email{wangcbo@zju.edu.cn }



\date{\today}

\begin{document}
\maketitle
\begin{abstract}
In this paper we study the initial boundary value problem for two-dimensional semilinear wave equations with small data, in asymptotically Euclidean exterior domains.
We prove that if $1<p\le p_c(2)$, the problem admits almost the same upper bound of the lifespan as that of the corresponding Cauchy problem, only with a small loss for $1<p\le 2$. It is interesting to see that the logarithmic increase of the harmonic function in $2$-D has no influence to the estimate of the upper bound of the lifespan for
$2<p\le p_c(2)$.
One of the novelties is that we can deal with the problem with flat metric and general obstacles (bounded and simply connected), and it will be reduced to the corresponding problem with compact perturbation of the flat metric outside a ball.
\end{abstract}
\tableofcontents

\section{Introduction}

In this paper, we are interested in the investigation of the blow-up part of the analogs of the Strauss conjecture in two dimensional asymptotically Euclidean exterior domain
$(\Om, \gm)$.
We assume $\pa\Om$ is a smooth Jordan curve.
By asymptotically Euclidean
exterior domain, we mean that
it is a submanifold of the asymptotically Euclidean space
$(\R^{2}, \gm)$.
For simplicity, we assume $\Om=\R^2\backslash \overline{ B_R}$ for some $R>0$, see however, Lemma \ref{thm-diffeo}.

Recall that,
for $(\R^{2}, \gm)$,
the metric $\gm$ is assumed to be of the form
\begin{align}
\label{dl1}
\gm=\gm_{1}+\gm_{2}\ ,
\end{align}
where $\gm_{1}$ is a spherically symmetric, long range perturbation of the flat metric $\gm_0$, and $\gm_{2}$ is a short range perturbation.
With possibly changing the choice of $R$, we could write $\gm_{1}$, in terms of the polar coordinates
$x=r(\cos\theta,\sin\theta)\in \Om$, as follows
\beeq
\label{dl3}
\gm_{1}=K^{2}(r)dr^{2}+r^{2}d\theta^{2}\ ,
\eneq
where $d\theta^2$ is the standard metric on the unit circle $\Sp^{1}$, and
\beeq
\label{dl2}
|\pa^{m}_{r}(K-1)|\les  \<r\>^{-m-\rho_1}
,  m=0, 1, 2,
\eneq
for some given constant $\rho_1 \in (0, 1]$.
Here and in what follows,
$\langle x\rangle=\sqrt{1+|x|^2}$, and
we use $A\les B$ ($A\gtrsim B$) to stand for $A\leq CB$ ($A\geq CB$) where the constant $C$ may change from line to line. Concerning $\gm_{2}$,  we have
$$\gm=g_{jk}(x)dx^j dx^k\equiv \sum^{2}_{j,k=1}g_{jk}(x)dx^j dx^k\ ,\ \gm_2=g_{2, jk}(x)dx^j dx^k\ ,$$
where we have used the convention that Latin indices $j$, $k$ range from $1$ to $2$ and the Einstein summation convention for repeated upper and lower indices.
Furthermore, we assume $\gm_2$ satisfies
\beeq
\label{eq-g2-ae2}
\nabla^\be g_{2,jk}=\mathcal{O}(\<r\>^{-\rho_2-|\be|}), |\be|\le 2\ ,
\eneq
for some $\rho_2>1$.
By these assumptions, it is clear that
there exists a constant $\delta_{0}\in(0, 1)$ such that
\beeq
\label{unelp}
\delta_{0}|\xi|^2\le
g^{jk}(x)\xi_{j}\xi_{k}\le
\delta_{0}^{-1} |\xi|^2, \forall  x, \xi\in\R^{2},\
K(r)\in (\delta_{0}, 1/\delta_{0})
\ ,
\eneq
where
$(g^{jk}(x))$ denotes the inverse of $(g_{jk}(x))$.

With these preparations in hand, we may write out our problem explicitly, that is,
initial boundary value problem of semilinear wave equations with small initial data posed on asymptotically Euclidean manifolds $(\Om, \gm)$ with $\gm$ satisfies \eqref{dl1}-\eqref{eq-g2-ae2}
  \begin{equation}\label{1}
\left \{
\begin{aligned}
&u_{tt}-\Delta_\gm u=|u|^p,~~~t>0,~~x\in \Omega,\\
&u(0, x)=\varepsilon u_0(x),~~ u_t(0, x)=\varepsilon u_1(x),~~x\in \Om,\\
&u(t,x)=0,\quad \ t>0,\ x\in \pa\Om,\\
\end{aligned} \right.
\end{equation}
where,
$\Delta_{\gm} =\nabla^j\partial_{j}$ is the standard Laplace-Beltrami operator, $\ep>0$ is a small parameter. Concerning the initial data, we assume
\beeq
\label{hs2}
(u_0, u_1)\in H^1_0(\Om)\times L^2(\Om),\quad \supp(u_0, u_1)\subset B_{R_{0}}\ ,
u_0, u_1\ge 0, u_0\not\equiv 0, u_1\not\equiv 0,
\eneq
 for some $R_{0}>R$.

Such kind of problem is the generalization of the Strauss conjecture (see \cite{Strauss}): the following Cauchy problem of semilinear wave equation with small initial data with sufficient regularity and sufficient decay at infinity
\begin{equation}\label{1a}
\left \{
\begin{aligned}
&u_{tt}-\Delta u=|u|^p,~~~(t, x)\in (0, T)\times \mathbb{R}^n,\\
&u(0, x)=\varepsilon u_0(x),~~ u_t(0, x)=\varepsilon u_1(x),~~x\in \mathbb{R}^n,\\
\end{aligned} \right.
\end{equation}
admits a critical exponent $p_c(n)(n\geq 2)$, which means that if $1<p\leq p_c(n)$ then problem \eqref{1a} has no global solution in general, whereas the solution exists globally in time if $p>p_c(n)$ and $0<\e \ll 1$.
Here $p_c(n)$ is the positive root of the quadratic equation
\begin{equation}\label{1b}
\begin{aligned}
(n-1)p^2-(n+1)p-2=0.\\
\end{aligned}
\end{equation}
This conjecture has been essentially verified, we list all of the corresponding results in the following table
(one can also find it in \cite{Takamura2}):
\begin{center}
\begin{tabular}{|c||c|c|c|}
\hline
$n$  & $1<p<p_c(n)$ & $p=p_c(n)$ & $p_c(n)<p\le 1+4/(n-1)$\\
\hline
\hline
2 &
\begin{tabular}{l}
Glassey \cite{Glassey1}
\end{tabular}
 & Schaeffer \cite{Schaeffer}  &
\begin{tabular}{l}
Glassey \cite{Glassey2}
\end{tabular}
\\
\hline
3 & John \cite{John}
& Schaeffer \cite{Schaeffer}  &
\begin{tabular}{l}
John \cite{John}
\end{tabular}
\\
\hline
$\ge4$ & Sideris \cite{Sideris} &
\begin{tabular}{l}
Yordanov-Zhang \cite{YorZh06},\\
 Zhou \cite{Zhou2}, indep.
\end{tabular}
&
\begin{tabular}{l}
Georgiev-Lindblad-Sogge \cite{Georgive}
\end{tabular}
\\
\hline
\end{tabular}
\end{center}

If there is no global solution, then it is interesting to estimate the time when the solution blows up, i.e., the lifespan. The results have been established for two cases: (I) subcritical power $(1<p<p_c(n))$; (II) critical power $(p=p_c(n))$. For the former case, we now know that there exist two positive constants $c$ and $C$
such that the lifespan satisfies for $n\ge 2$ and
$\max(1,2/(n-1))<p<p_c(n)$
\begin{equation}\label{1c}
\begin{aligned}
c\varepsilon^{\frac{-2p(p-1)}{\gamma(n, p)}}\leq T(\varepsilon)\leq C\varepsilon^{\frac{-2p(p-1)}{\gamma(n, p)}},
\end{aligned}
\end{equation}
where $\gamma(n, p)=2+(n+1)p-(n-1)p^2>0$. We may read the facts from the following table:
\begin{center}
\begin{tabular}{|c||c|c|c|}
\hline
$n$  & Lower bound & Upper bound  \\
\hline
\hline
2 &
\begin{tabular}{l}
Zhou \cite{Zhou3}
\end{tabular}
 &
\begin{tabular}{l}
Zhou \cite{Zhou3}
\end{tabular}
\\
\hline
3 & Lindblad \cite{Lindblad3}
&Lindblad \cite{Lindblad3}
\\
\hline
$\ge4$ & Lai-Zhou \cite{Lai} &
\begin{tabular}{l}
Takamura \cite{Takamura2}
\end{tabular}
\\
\hline
\end{tabular}
\end{center}
For relatively small power in dimension $2$, we have richer results. For $(n, p)=(2, 2)$, Lindblad \cite{Lindblad3} obtained the following result \begin{equation}
\left\{
\begin{array}{ll}
\exists \lim\limits_{\e\rightarrow 0^{+}} a(\e)^{-1}T(\e)>0, & \mathrm{if}~\int_{\R^2}u_1(x)dx\neq 0,\\
\exists \lim\limits_{\e\rightarrow 0^{+}} \e T(\e)>0,& \mathrm{if}~\int_{\R^2}u_1(x)dx= 0,
\end{array}
\right.
\end{equation}
where $a(\e)$ denotes a number satisfying 
\[
a^2\e^2\log(1+a)=1.
\]
For $1<p<2$ and $n=2$, due to the results in Takamura \cite{Takamura2} and Imai et al. \cite{Takamura6}, we have
\begin{equation}
\left\{
\begin{array}{ll}
c\e^{-\frac{p-1}{3-p}}\le T(\e)\le C\e^{-\frac{p-1}{3-p}},&\mathrm{if}~\int_{\R^2}u_1(x)dx\neq 0,\\
c\e^{-\frac{2p(p-1)}{\gamma(2, p)}}\le T(\e)\le C\e^{-\frac{2p(p-1)}{\gamma(2, p)}},&\mathrm{if} ~\int_{\R^2}u_1(x)dx= 0.
\end{array}
\right.
\end{equation}

For the critical case $(p=p_c(n))$, the lifespan is much longer
and has the form:
\begin{equation}\label{1d}
\begin{aligned}
\exp(c\varepsilon^{-p(p-1)})\leq T(\varepsilon)\leq \exp(C\varepsilon^{-p(p-1)}).
\end{aligned}
\end{equation}
We have the following table:
\begin{center}
\begin{tabular}{|c||c|c|c|}
\hline
$n$  & Lower bound & Upper bound  \\
\hline
\hline
2 &
\begin{tabular}{l}
Zhou \cite{Zhou3}
\end{tabular}
 &
\begin{tabular}{l}
Zhou \cite{Zhou3}
\end{tabular}
\\
\hline
3 &Zhou \cite{Zhou5}
&Zhou \cite{Zhou5}
\\
\hline

$\ge4$ &
\begin{tabular}{l}
Lindblad-Sogge \cite{Lindblad2} for $n\leq 8$ \\
or radial solution
\end{tabular}
&
\begin{tabular}{l}
Takamura-Wakasa \cite{Takamura}
\end{tabular}
\\
\hline
\end{tabular}
\end{center}

As waves propagate to infinity of the space, hence besides the Cauchy problem in the whole space, it is also interesting to consider the corresponding obstacle problem, i.e., the initial boundary value problem in exterior domain. Due to the difficulty caused by the boundary, such kind of problem has not been well understood, particularly for the global existence in high dimensions ($n\ge 5$). Anyway, we have the following results of global existence vs blow-up:
\begin{center}
\begin{tabular}{|c||c|c|c|}
\hline
$n$  & $1<p<p_c(n)$ & $p=p_c(n)$ & $p>p_c(n)$ \\
\hline
\hline
2 &
\begin{tabular}{l}
Li-Wang \cite{Li}
\end{tabular}
 & Lai-Zhou \cite{Lai1} &
\begin{tabular}{l}
Smith-Sogge-Wang \cite{Smith}
\end{tabular}
\\
\hline
3 & Zhou-Han \cite{Zhou7} & Lai-Zhou \cite{Lai2}
&
\begin{tabular}{l}
Hidano et al \cite{Hidano}
\end{tabular}
\\
\hline
4 & Zhou-Han \cite{Zhou7} & Sobajima-Wakasa \cite{SW}  &
\begin{tabular}{l}
Du et al \cite{Du},
reproved by\\
Hidano et al \cite{Hidano}
\end{tabular}
\\
\hline
$\ge5$ & Zhou-Han \cite{Zhou7} &
\begin{tabular}{l}
Lai-Zhou \cite{Lai3},\\
reproved by\\
Sobajima-Wakasa \cite{SW}
\end{tabular} &
\begin{tabular}{l}
$p=2$, Metcalfe-Sogge \cite{Metcalfe},\\
reproved by Wang \cite{Wang2}
\end{tabular}
\\
\hline
\end{tabular}
\end{center}

Just like the Cauchy problem, it is meaningful to study the lifespan for the blow-up exponent. We expect the same estimate as that of the Cauchy problem, regardless of the boundary obstacle, at least when the obstacle is nontrapping.
Denoting the expected sharp lower bound and upper bound by \lq\lq L" and \lq\lq U" respectively, we have the following known results:
\vskip5pt
\begin{center}
\begin{tabular}{|c||c|c|}
\hline
$n$  & $1<p<p_c(n)$ & $p=p_c(n)$\\
\hline
\hline
2 &
\begin{tabular}{l}
L : {\bf ?}\\
U : {\bf This work}\\
\end{tabular}
&
\begin{tabular}{l}
L : {\bf ?}\\
U :  {\bf This work}
\end{tabular}
\\
\hline
3 &
\begin{tabular}{l}
L : Du-Zhou \cite{Du1}$(p=2)$,\\
~~~~ Yu \cite{Yu} $(2<p<1+\sqrt2)$\\
~~~~~Wang \cite{Wang2}$(2\leq p<1+\sqrt2)$\\
U : Zhou-Han \cite{Zhou7}
\end{tabular}
&
\begin{tabular}{l}
L : Yu \cite{Yu}
~~~~$(T(\varepsilon)\geq \exp(c\varepsilon^{-\sqrt2}))$,\\
~~~~Improved by Wang \cite{Wang2}\\
~~~~\ \ $(T(\varepsilon)\geq \exp(c\varepsilon^{-2\sqrt2}))$\\
U :  Lai-Zhou \cite{Lai2},\\
~~~~~reproved by
Sobajima-Wakasa \cite{SW}
\end{tabular}
\\
\hline
4 &
\begin{tabular}{l}
L : {\bf ?}\\
U : Zhou-Han \cite{Zhou7}
\end{tabular}
&
\begin{tabular}{l}
L : Zha-Zhou \cite{Zha}\\
~~~~~reproved by Wang \cite{Wang2}\\
U :Sobajima-Wakasa \cite{SW}
\end{tabular}
\\
\hline
$\ge5$ &
\begin{tabular}{l}
L : {\bf ?}\\
U : Zhou-Han \cite{Zhou7}
\end{tabular}
&
\begin{tabular}{l}
L : {\bf ?} \\
U :  Lai-Zhou \cite{Lai3},\\
~~~~~reproved by
Sobajima-Wakasa \cite{SW}
\end{tabular}

\\
\hline
\end{tabular}
\end{center}
\vskip5pt
We also have to mention the generalization of problem \eqref{1a} from Euclidean space to other manifold, such as asymptotically Euclidean manifolds (see \cite{MW}, \cite{SWa}, \cite{WY1} and references therein), and black hole spacetime (see \cite{CG}, \cite{LLM}, \cite{LMS}, \cite{MMT} and references therein).
One can find a detailed description of such kind of generalization in a recent survey paper \cite{Wang1}. Another direction is to consider the initial boundary value problem in exterior domain with big initial data (see \cite{LZ}, \cite{SS} and references therein).

 Before stating our results, we shall make a hypothesis.\\
  {\bf Hypothesis:}
There exists  $\la_0\in(0,1/(2R))$, such that for any $\la\in (0,\la_0)$,  we could find a solution $\phi_\la$ solving
\beeq\label{eq-eigen}
\Delta_{\gm}\phi_\la=\la^{2}\phi_\la \ , x\in  \Om,
\ \phi_\la|_{\pa\Om}=0\ ,
\eneq
which enjoys the following uniform estimates (independent of $\la$)
\beeq\label{H1}\tag{H}
\phi_\la\sim \left\{
\begin{array}{ll}
	\frac{\ln r/R}{\ln (R\la)^{-1}} & r\le \la^{-1},\\
	\<r\la\>^{-1/2}e^{\la\int^{r}_{R}K(\tau)d\tau} & r\ge \la^{-1},
\end{array}
\right.
\eneq
for any $\la\in (0, \la_0)$.

\begin{thm}\label{thm1}
Let $p\in (1, p_c(2)]$ and assume  \eqref{H1}. Consider the problem \eqref{1} with initial data \eqref{hs2}, posed on asymptotically Euclidean exterior domain $(\Om,\gm)$ satisfying \eqref{dl1}-\eqref{eq-g2-ae2}. Then we have the following \begin{enumerate}
  \item
There exists a unique weak solution  $u(t, x)\in C_{t} H^1_0\cap C_{t}^1 L^2$ to the initial boundary value problem \eqref{1} on
$[0, T(\ep))\times \Om$, where $T(\ep)$ denotes the lifespan, i.e., the maximal time of existence.
  \item There exist constants $C,\e_0>0$ independent of $\e$
such that
for any $\e\in(0,\e_0)$, we have
\begin{equation}
\label{lifespan1}
T(\ep)\leq
\left\{
\begin{array}{ll}
\exp\left(C\e^{-p(p-1)}\right),& p=p_c(2),\\
C\e^{-\frac{2p(p-1)}{2+3p- p^2}}, & 2\leq p<p_c(2),\\
C
(\varepsilon^{-1}\ln (\varepsilon^{-1}))^{(p-1)/(3-p)}
, & 1<p<2.
\end{array}
\right.
\end{equation}
\end{enumerate}
\end{thm}
\begin{rem}
Comparing with the upper bounds with that of the corresponding Cauchy problem,
we see that we have the same upper bounds of the lifespan, when $2<p\le p_c(2)$ (with log loss for the $1<p\le 2$).
It will be interesting to determine the sharp estimate of the lifespan.
We expect that our upper bound is sharp, at least when the obstacle is nontrapping and $2<p\le p_c(2)$.
\end{rem}

The local well-posed result follows from a standard energy argument. For the upper bound of lifespan estimates, it relies on the so-called test function method. For the subcritical case, we have to show the proper asymptotic behaviors of two test functions, that is, \eqref{asyphi0} for $\phi_0$ and \eqref{H1} for $\phi_{\lambda_1}$ with some fixed $\lambda_1\in (0, \lambda_0)$. However, for the critical case, we need to use a family of test function $\phi_{\lambda}$ with $\lambda$ varying in $(0, \lambda_0)$, to construct another test function $b_q$ with more subtle asymptotic behavior as stated in \eqref{bq3} and \eqref{bqge3} below. Concerning $\phi_0$, we have
\begin{lem}
\label{phi_0}
Let $(\Om, \gm)$ be an asymptotically Euclidean exterior domain
 satisfying \eqref{dl1}-\eqref{eq-g2-ae2}.
 Then there exists a solution $\phi_0$ to
\beeq
\label{test1}
\Delta_{\gm}\phi_0=0 \ , x\in \Om\ , \
\phi_{0}|_{\pa\Om}=0
\eneq
satisfying
\beeq
\label{asyphi0}
\phi_0(x)\simeq \ln r / R\ .
\eneq
\end{lem}

\begin{rem}
It will be clear from the proof that the first two upper bounds in \eqref{lifespan1} hold for non-negative data with either $u_{0}$ or $u_{1}$ nontrivial. While, for the last upper bound, we need only to assume
non-negative data with
$u_{1}\neq 0$.
\end{rem}

Noting that Theorem \ref{thm1} holds under the hypothesis \eqref{H1}. Let us review some cases where the assumption \eqref{H1} is valid. It is well known for the Euclidean space $(\R^2, \gm_0)$, where
$\phi_\la$ could be given by the spherical average of $e^{\la x\cdot\om}$,
$$\phi_\la(x)=\int_{\mathcal{S}^{1}}e^{\la x\cdot\omega}d\omega\sim
 \<r\la\>^{-1/2}e^{\la r}
\ ,$$
see Yordanov-Zhang \cite{YorZh06}.
When $\gm_3$ is  an exponential perturbation, that is, there exists $\al>0$ so that
\beeq
\label{dlfjia}
|\nabla g_{3,jk}|+|g_{3,jk}|\les  e^{-\al\int^{r}_{R}K(\tau)d\tau}\ ,
\eneq
 the corresponding estimate for ($\R^2, \gm_1+\gm_3$)
is recently obtained
for $\gm=\gm_1+\gm_3$ by Liu-Wang \cite{LWp}, while the case
 $\gm=\gm_0+\gm_3$
was obtained by
Wakasa-Yordanov in \cite{WaYo19}.
Based on these results, we could verify the hypothesis \eqref{H1}, in the case of $\gm=\gm_{1}+\gm_{3}$,
and thus obtain the following

\begin{thm}\label{thm2}
Let $\gm=\gm_1+\gm_3$. Then the hypothesis \eqref{H1} holds and we have the same results as that in Theorem \ref{thm1}.
\end{thm}

In the above,
we have considered exclusively on the asymptotically Euclidean exterior domains to disk.
It is natural to ask what happens for the general
asymptotically Euclidean exterior domains $(D, \gm)$.
It turns out that,
when $\pa D$ is a smooth Jordan curve,
the general problem could be reduced to the case of disk exterior.
Actually, with the help of the Riemann mapping theorem and a cut-off argument, we could prove the following
\begin{lem}\label{thm-diffeo}
Let $D\subset\R^{2}$ be exterior domain to
a smooth Jordan curve  $\pa D$. Then there exists a  diffeomorphism mapping preserving the boundary
\beeq
\label{eq-diffeo}\mathcal{A}: D\to \R^{2}\backslash \overline{ B_{R}}\ ,\eneq
for some $R>0$,
which is an identity map for $r\gg 1$.
\end{lem}
With the help of this lemma, we could translate the problem
\eqref{1} in the domain $(D,\gm)$, into the corresponding problem for $(\Om, \tilde{\gm})$, by change of variables with the new unknown function  $w(x)=u(\mathcal{A}^{-1}(x)):\Om\to \R$. Thus, whenever we have
\eqref{H1} for $(\Om,\tilde{\gm})$,
the lifespan estimates
 \eqref{lifespan1} apply for  $(D,\gm)$.

In particular, the result applies for
any exterior domains  to
smooth Jordan curves.

 More precisely, let $\mathcal{K}\subset \R^2$ be a domain interior to a smooth Jordan curve, (that is, $\mathcal{K}$ is a nonempty, open, bounded, smooth, simply connected domain), we consider the following 
  nonlinear wave equations with obstacle $\mathcal{K}$
\begin{equation}\label{nlw2}
\left \{
\begin{aligned}
&u_{tt}-\Delta u=|u|^p,~~~t>0,~~x\in D=\R^2\backslash \bar{ \mathcal{K}},\\
&u(0, x)=\varepsilon u_0(x),~~ u_t(0, x)=\varepsilon u_1(x),~~x\in D,\\
&u|_{\partial   \mathcal{K} }=0\ .\\
\end{aligned} \right.
\end{equation}
With the help of the transform $w(x)=u(\mathcal{A}^{-1}(x)):\Om\to \R$, we need only to consider the  corresponding problem for $(\Om, \tilde{\gm}_0)$.
Thanks to Lemma \ref{thm-diffeo}, $\tilde{\gm}_0$ is a compact perturbation of $\gm_0$, which allows us to apply
Theorem \ref{thm2} to obtain the following result.
\begin{thm}\label{thm3}
Let  $\mathcal{K}\subset \R^2$ be nonempty bounded smooth simply connected domain.
Consider \eqref{nlw2} with initial data \eqref{hs2}. Then we have the same
lifespan estimates
 \eqref{lifespan1} for energy solutions.
\end{thm}
\begin{rem}
When the spatial dimension is not greater than $4$,
all of the previous blow-up results and lifespan estimates for exterior problem with critical power heavily rely on the assumption that
the obstacle is a  ball
(see  \cite{Lai2},  \cite{Lai1}, \cite{SW}), under which they can construct some special test functions explicitly. In contrast, our results hold for very general obstacle in $2$-D.
\end{rem}

\subsubsection*{Outline} Our paper is organized as follows. In the next section, we sketch the proof of local well-posedness for the energy solutions, by a  standard energy argument. In particular,
it shows the finite speed of propagation \eqref{thm1step1a}, for the solution.
In Section \ref{sec:1stTest},
 we prove the existence of test function $\phi_{0}$, Lemma \ref{phi_0}, for
  the Dirichlet problem of the Laplace equation on $(\Om, \gm)$.
With the help of $\phi_{0}$, as well as the hypothesis \eqref{H1}, in Section \ref{sec:AE4}, we present  the proof of Theorem \ref{thm1}.
In Section \ref{sec:AE3}, we show that the hypothesis \eqref{H1} holds, at least when the metric $\gm$ is exponential perturbation of a spherically symmetric, long range asymptotically Euclidean metric $\gm=\gm_1+\gm_3$.
At last, in Section \ref{sec:diffeo},   with the help of the Riemann mapping theorem, we prove Lemma \ref{thm-diffeo}, which enables us to reduce the problems with general obstacles to the problem exterior to a disk (keeping the metric near the spatial infinity).

\section{Local well-posedness}
Before going to the proof of blow-up results, we first show the local
well-posedness for problem \eqref{1}, based on energy estimates.
By multiplying the equation in \eqref{1} with $\partial_tu$, we get the energy estimate
\begin{equation}\label{energyineq}
\begin{aligned}
\frac12\frac{d}{dt}\int_{\Omega}\left(u_t^2+g^{ij}(x)u_{x_i}u_{x_j}\right)dV_{\gm}=\int_{\Omega}|u|^pu_tdV_{\gm},\\
\end{aligned}
\end{equation}
where
$dV_\gm=\sqrt{\det (\gm_{jk}(x))} dx$ is the volume form with respect to the metric.
Due to the assumption of $g^{ij}(x)$ in \eqref{unelp}, this implies
\begin{equation}\label{energyineq1}
\begin{aligned}
\|\partial u\|_{L^{\infty}L^2([0, T)\times \Omega)}\les \ \e+T\||u|^p\|_{L^{\infty}L^2([0, T)\times \Omega)}.\\
\end{aligned}
\end{equation}
On the other hand, it is easy to get
\begin{equation}\label{energyineq2}
\begin{aligned}
\|u\|_{L^{\infty}L^2([0, T)\times \Omega)}&\les \ \|u_0\|_{L^2(\Omega)}+T\|u_t\|_{L^{\infty}L^2([0, T)\times \Omega)}\\
&\les \ \e(1+T)+T^2\||u|^p\|_{L^{\infty}L^2([0, T)\times \Omega)}.\\
\end{aligned}
\end{equation}
It then follows by combining \eqref{energyineq1} and \eqref{energyineq2} that
\begin{equation}\label{energyineq3}
\begin{aligned}
\|u\|_{L^{\infty}H^1([0, T)\times \Omega)}\les \ \e(1+T)+(T+T^2)\||u|^p\|_{L^{\infty}L^2([0, T)\times \Omega)}.\\
\end{aligned}
\end{equation}
Recalling that
\[
H^1(\Omega)\hookrightarrow L^{2p}(\Omega),~~~1<p<\infty,
\]
with the help of all these estimates, a standard contraction mapping argument yields the desired local well-posed result.

In particular, as the solution is obtained through iteration,
we see that the solution enjoys the
 finite speed of propagation:
\begin{equation}\label{thm1step1a}
\supp \ u(t)\subset \left\{x\in\Om: \int_R^{r}K(\tau)d\tau\le t+R_1\right\}
:=D_1(t)\subset\left\{x: r\le \frac{t+R_{2}}{\delta_0}\right\},
\eneq
for any data satisfying \eqref{hs2},  where $ R_{2}> R_1=\int_{R}^{R_{0}}K(\tau)d\tau>0$.

\section{Test function for the Dirichlet problem}\label{sec:1stTest}
In this section, we consider the Laplace equation posed on the asymptotically Euclidean exterior domain $(\Om, \gm)$, that is,
Lemma \ref{phi_0}.

We observe that, for the purpose of the lemma, we need only to prove the
following weaker estimate
$$\phi_0(x)\simeq \ln r, \ r\gg 1\ .$$
  Actually,
  by the strong maximum principle,
we know that  $\phi_0(x)$ is positive everywhere in $\Om$.
Moreover, by  Hopf's lemma, we have
 $$\pa_r\phi_0(r,\theta)|_{r\rightarrow R_{+}}>0\ ,$$
 which gives us that
 $$\phi_0(r,\theta)\sim r-R\sim \ln  {r}/R\ ,$$
 for any $R\le r\les 1$.

\subsection{Kelvin transformation}
At first,
inspired by the classical Kelvin transform,
 we use the spatial inversion  to introduce the conformal
compactization. More specifically, we introduce new coordinate in the
${\bar{B}_{1/R}}\backslash\{0\}$ for $\Om$
\beeq\label{ktrans}
\Phi(x)=\phi_0( |x|^{-2}x)=\phi_0(x^*)\ , x\in {\bar{B}_{1/R}}\backslash\{0\}\ .\eneq
Let $ds^2=\gm_{jk}(x^{*})dx^{*j}dx^{*k}=\tilde{h}_{jk}(x)dx^jdx^k$, where $x=r\om$, $\gm_{jk}(x^{*})=\de_{jk}+f(x^{*})\om_j\om_k+\gm_{2,jk}(x^{*})$ and $f(x^{*})=K^2(x^{*})-1$. By \eqref{dl2}-\eqref{eq-g2-ae2}, we have
\beeq
\label{jian1}
f(x^{*})=\mathcal{O}(|x|^{\rho_1})
, \gm_{2,jk}(x^{*})=\mathcal{O}(|x|^{\rho_2}), \ \rho_1\in (0, 1], \rho_2>1\ ,
\eneq
as $x\to 0$.
As$$\frac{\pa x^{*j}}{\pa x^k}=\frac{\de^{jk}-2\omega^j\omega^k}{r^2}
\ ,$$
we get
$$\tilde{h}_{jk}(x)=\frac{1}{r^4}
(\de^{ji}-2\omega^j\omega^i)
\gm_{il}(x^{*})
(\de^{lk}-2\omega^l\omega^k):=\frac{h_{jk}}{r^4}\ ,$$
where
\begin{align}
\nonumber
h_{jk}(x)&=(\de^{ji}-2\om^j\om^i)\gm_{il}(x^{*})(\de^{lk}-2\om^{l}\om^{k})\\ \label{htrans}
&=\de_{jk}+f(x^{*})\om_j\om_k+\gm_{2,il}(x^{*})(\de^{ji}-2\om^j\om^i)(\de^{lk}-2\om^{l}\om^{k})\ .
\end{align}
Then we obtain
\beeq\label{ktrans2}
\De_{\gm}\phi_0(x^{*})=\De_{\tilde{h}}\Phi(x)=r^{4}\De_{h}\Phi(x)\ .
\eneq
In particular, if we have $\gm=\gm_1$, there is $h_0$ such that
\beeq
\label{ktans3}
\De_{\gm_1}\phi(x^{*})=r^{4}\De_{h_0}\Phi_0(x)\ , h_{0,jk}(x)=\de_{jk}+f(x^{*})\om_j\om_k\ .
\eneq

\subsection{Test function for $\gm_1$}
Considering
$$\Delta_{\gm_1} \phi=0, x\in  \Om;\ \phi=0, x\in \pa \Om\ ,
$$
it is easy to find a
spherically symmetric
solution
\beeq\label{eq-key-phi0}\phi(r)=\int_R^r \frac{K(s)}{s}ds\simeq \ln \frac rR\ ,\eneq
as
\beeq\label{eq-gm1}
\Delta_{\gm_1}=K^{-1}r^{-1}\pa_{r}  K^{-1} r\pa_{r}+r^{-2}\pa_\theta^2 \ .\eneq

By \eqref{ktans3}, we have
\beeq
\label{jian2}
\Phi_0(x)=\phi(\frac{x}{|x|^2})=\phi(\frac{1}{r})=\int_R^{\frac{1}{r}} \frac{K(s)}{s}ds\eneq
satisfying
\beeq
\begin{cases}
\De_{h_0}\Phi_0=0\ , x\in B_{1/R}\backslash \{0\},\\
\Phi_0=0\ , x\in \pa B_{1/R}\ .
\end{cases}
\eneq

\subsection{Proof of Lemma \ref{phi_0}}By \eqref{ktrans}-\eqref{ktrans2}, we are reduced to finding a solution in the region
$B_{1/R}\backslash\{0\}$
\beeq
\begin{cases}
\De_{h}\Phi=0\ , x\in B_{1/R}\backslash\{0\}, \\
\Phi=0\ , x\in \pa B_{1/R}
\end{cases}
\eneq
satisfying $\Phi(x)\sim \ln \frac{1}{r}$ near $r=0$. Let $u(x)=\Phi(x)-\Phi_0(x)$, it remains to construct a solution $u\in L^\infty(B_{1/R}\backslash\{0\})$, due to the fact that $\Phi_0\sim \ln (1/r)$ near $r=0$. Concerning $u$, it satisfies
\beeq
\nonumber
\begin{cases}
\De_{h}u=(\De_{h_0}-\De_{h})\Phi_0\ , x\in B_{1/R}\backslash\{0\} \ ,\\
u=0\ , x\in \pa B_{1/R}\ ,
\end{cases}
\eneq
and we would like to view it as the Dirichlet problem in the ball
$B_{1/R}$.

For that purpose, we  set $h_{jk}(0)=\de_{jk}$ so that it is continuous in $B_{1/R}$, in view of
\eqref{jian1}-\eqref{htrans}.
Moreover, we claim that
\beeq
\label{claim1}
(\De_h-\De_{h_0})\Phi_0=\mathcal{O}(r^{\rho_2-2})\ ,
\eneq
which could be written in the form $\pa_{1}F$ for some $F=\mathcal{O}(r^{\rho_2-1})\in L^{\infty}(B_{1/R})$.
With the help of the claim, by standard elliptic existence theorems, there is a unique solution $u\in H_{0}^{1}(B_{1/R})$. In addition,
as the equation is of divergence form,
$$\pa_j (|h|^{1/2}h^{jk} \pa_k u)=|h|^{1/2}
(\De_h-\De_{h_0})\Phi_0=\pa_{1}F\ ,$$
with $F\in  L^{\infty}(B_{1/R})$,
an application of Meyer's theorem (see, e.g., Taylor \cite[Chapter 14, Proposition 12.2]{T11}) gives us
$u\in W^{1,q}(B_{1/R})$ for some $q>2$, which in turn gives us the desired result
$u\in L^\infty(B_{1/R})$.

We are left to give the proof of the claim \eqref{claim1}.
By calculation, we have
$$(\De_h-\De_{h_0})\Phi_0=
(h^{jk}-h_{0}^{jk})
\pa_{j}\pa_{k}\Phi_0+
[|h|^{-1/2}\pa_j (|h|^{1/2}h^{jk})
-|h_0|^{-1/2}\pa_j (|h_0|^{1/2}h_{0}^{jk})]
 \pa_k\Phi_0\ .
$$
By \eqref{jian2} and \eqref{dl2},  we get
$$\pa_j\Phi_0=\mathcal{O}(r^{-1}), \pa_j\pa_k\Phi_0=\mathcal{O}(r^{-2})\ .$$
Similarly, by  \eqref{dl2}-\eqref{eq-g2-ae2}, \eqref{htrans} and \eqref{ktans3}, we have
$$
|\nabla^\be (h_{jk}-h_{0,jk})|
+|\nabla^\be (h^{jk}-h_0^{jk})|
+|\nabla^\be (|h|-|h_{0}|)|
=\mathcal{O}(r^{\rho_2-|\be|}), |\be|=0,1\ .$$
In summary,
it is easy to see that
$$|(\De_h-\De_{h_0})\Phi_0|\les \sum_{j,k}
\sum_{1\le |\be|\leq 2, |\al|+|\be|=2}
|\nabla^{\al}\gm^{jk}_{2}(x^{*})||\nabla^{\be}\Phi_0|=\mathcal{O}(r^{\rho_2-2})\ ,$$
which completes the proof.

\subsection{An integral estimate}
With $\phi_0$ and its asymptotic behavior,
together with \eqref{H1} for $\phi_{\la}$,
 we will need the following
\begin{lem}
\label{asyphi0la}
Let $\la_{1}\in(0,\la_0)$, then we have
\begin{equation}\label{phi0la}
\begin{aligned}
\int_{\Omega\cap (\int_R^rK(s)ds\le t+R_1)}\phi_0^{-\frac{1}{p-1}}e^{-p'\la_1 t}\phi_{\la_1}^{p'}dV_\gm\lesssim \left(\ln(t+1)\right)^{-\frac{1}{p-1}}(t+1)^{1-\frac{p'}{2}},
\end{aligned}
\end{equation}
where $p'=\frac{p}{p-1}$ and the implicit constant may depend on $\la_{1}$.
\end{lem}
\begin{proof}
Let us begin with the region with
$r\le \la_1^{-1}$,
in that case,  by \eqref{asyphi0} and the first estimate in \eqref{H1},
we have
$\phi_0^{-\frac{1}{p-1}}\phi_{\la_1}^{p'}\les \phi_{0}(r)
(\phi_0(\la_1^{-1}))^{-\frac{p}{p-1}}\les \ln(r/R)$ and
 the integral is controlled by
$$\int_{R\le r\le \la_1^{-1}}\phi_0(r)e^{-p'\la_1t}rdr\
\les\  e^{-p'\la_1t}.$$

For the remained case
$r\ge \la_1^{-1}$,
we could use \eqref{asyphi0} and the second estimate in \eqref{H1}.
Thus we have
$$\int_{\{r\ge \la_1^{-1}\}\cap \{\int_R^rK(s)ds\le \frac{t+R_1}{2}\}}\phi_0^{-\frac{1}{p-1}}e^{-p'\la_1 t}\phi_{\la_1}^{p'}dV_\gm\les \int_{r\les t+R_2} e^{-\frac{p'\la_1t}2} rdr
\les e^{\frac{-p'\la_1t}{2}}(t+1)^2\ ,$$
and
\begin{equation}\label{I3}
\begin{aligned}
&\int_{\frac{t+R_1}{2}\le\int_R^rK(s)ds\le t+R_1}\phi_0^{-\frac{1}{p-1}}e^{-p'\la_1 t}\phi_{\la_1}^{p'}dV_\gm\\
\les&\left(\ln(t+1)\right)^{-\frac{1}{p-1}}\int_{\frac{t+R_1}{2}\le\int_R^rK(s)ds\le t+R_1}e^{-p'\la_1 t+p'\la_1\int_R^rK(s)ds}\langle\la_1r\rangle^{1-\frac{p'}{2}}dr\\
\les&\left(\ln(t+1)\right)^{-\frac{1}{p-1}}(t+1)^{1-\frac{p'}{2}}.
\end{aligned}
\end{equation}
In summary, we obtain  \eqref{phi0la}.
\end{proof}

\section{Proof of Theorem \ref{thm1}}\label{sec:AE4}

In this section, we are devoted to the proof for Theorem \ref{thm1}, under the assumption \eqref{H1}.

\subsection{
Subcritical case
}

Let $\eta(t)\in C^{\infty}([0, \infty))$ be a decreasing function  satisfying
\[
\eta(t)=
 \left\{
 \begin{array}{ll}
 1 &  t\le\frac12,\\
 0 &  t\ge 1,\\
 \end{array}
 \right.
\eta^*(t)=\eta(t)\chi_{[1/2, 1]}(t)
\] and
\[
\eta_T(t)=\eta(t/T),
\eta^*_T(t)=\eta^*(t/T),
\ T\in (1, T(\ep)).
\]

As $u\in C_{t}H^1_0\cap C_{t}^1 L^2$ is the energy solution to
 \eqref{1},
 we have
 finite speed of propagation \eqref{thm1step1a}
and
  $u\in C_{t}^2 H^{-1}$
based on the equation.
Using
 $\eta_T^{2p'}(t)\phi_0(x)$
 as a test function, where $\phi_0(x)$ is the test function in Lemma \ref{phi_0}, we obtain
\begin{eqnarray*}
&&\<|u|^p, \eta_T^{2p'}(t)\phi_0(x)\>_\gm \\
& = & \<u_{tt}-\Delta_\gm u, \eta_T^{2p'}(t)\phi_0(x)\>_\gm \\
 & = &
\frac{d}{dt}( \<u_t, \eta_T^{2p'}\phi_0(x)\>_\gm
-\<u, \pt \eta_T^{2p'}\phi_0(x)\>_\gm)+\<u, \Box_\gm  \big(\eta_T^{2p'}(t)\phi_0(x)\big)\>_\gm
\\
 & = &
\frac{d}{dt}( \<u_t, \eta_T^{2p'}\phi_0(x)\>_\gm
-\<u, \pt \eta_T^{2p'}\phi_0(x)\>_\gm)+\<u,\pt^2 \eta_T^{2p'}(t)\phi_0(x)\>_\gm
\end{eqnarray*}
 where
$\Box_\gm= \pt^2-\Delta_\gm$,
 $\<u,v\>_\gm$ denotes the dual relation between distribution $u$ and test function $v$, which agrees with
 $\int_{\Om} u(x)v(x) \sqrt{\det (\gm_{jk}(x))} dx
 =\int_{\Om} u(x)v(x) dV_\gm
 $ for usual functions $u,v$.
  Integrating over $[0, T]$, and observing that
  $\eta_T(T)=  \pt\eta_T(0)=
  \pt\eta_T(T)=0$,
  we get
\beeq\label{thm1step1}\int_0^T\int_{\Omega}|u|^p\eta_T^{2p'}\phi_0dV_\gm dt
+\e \int_{\Omega}u_1\phi_0dV_\gm=\int_0^T\int_{\Omega}u\partial_t^2\eta_T^{2p'}\phi_0dV_\gm dt\ .
\eneq


Noting that
\[
\partial_t^2\eta_T^{2p'}=2p'(2p'-1)\eta_T^{2p'-2}(\partial_t\eta_T)^2
+2p'\eta_T^{2p'-1}(\partial_t^2\eta_T),\]
since $(\partial_t\eta_T)^2, |\partial_t^2\eta_T|\le CT^{-2}$ for some $C>0$, and $\eta_T\le 1$, then we have
\[
\partial_t^2\eta_T^{2p'}
=\CO(T^{-2}\eta_T^{2p'-2}).
\]
Thus we obtain
\begin{eqnarray*}
 &  & \int_0^T\int_{\Omega}u\partial_t^2\eta_T^{2p'}\phi_0 dV_\gm dt\\
 & \le  & C\left(\int_0^T\int_{\Omega}|u|^p\eta_T^{2p'}\phi_0 dV_\gm dt\right)^{\frac 1p}\left(\int_{\frac T2}^T\int_{ D_1(t)}
 T^{-2p'}
 \phi_0 dV_\gm dt\right)^{\frac {1}{p'}}\\
&\le &CT^{3-2p'}\ln T+\frac13\int_0^T\int_{\Omega}|u|^p\eta_T^{2p'}\phi_0 dV_\gm dt,
\end{eqnarray*}
where we have used H\"{o}lder's inequality, \eqref{asyphi0} and Young's inequality.
Plugging it to  \eqref{thm1step1}, we know that
\begin{equation}\label{thm1step3}
\e \int_{\Omega}u_1(x)\phi_0 dV_\gm +\int_0^T\int_{\Omega}|u|^p\eta_T^{2p'}\phi_0 dV_\gm dt\le CT^{3-2p'}\ln T.
\end{equation}
This inequality implies the lifespan estimate for $1<p<2$ in \eqref{lifespan1}. Actually, from the last inequality we conclude that
 $T^{3-2p'}\ln T\gtrsim \e$ for any $T\in (1,T(\ep))$, which gives us, for
 $1<p<3$,
\begin{equation}\label{lifespan12}
T(\e)\les\ (\varepsilon^{-1}\ln (\varepsilon^{-1}))^{(p-1)/(3-p)}\ .
\end{equation}
 But, comparing to the one
\begin{equation}\label{lifespan1pc2}
\begin{aligned}
T(\e)\les\ \e^{-\frac{2p(p-1)}{2+3p- p^2}},
\end{aligned}
\end{equation}
which we will get for $1<p<p_c(2)$, the lifespan estimate \eqref{lifespan12} is better than \eqref{lifespan1pc2} only for $1<p<2$.

To proceed, we introduce another test function
\[
\eta_T^{2p'}\psi(t, x)=\eta_T^{2p'}e^{-\la_1t}\phi_{\la_1}(x)
\]
to get
\begin{eqnarray*}
&&\int_0^T\int_{\Omega}|u|^p\eta_T^{2p'}\psi(t, x) dV_\gm dt\\
&=&\int_0^T\<u_{tt}-\Delta_\gm u, \eta_T^{2p'}\psi(t, x)\>_\gm dt,\\
&=&
\left.\left(\<u_t, \eta_T^{2p'}\psi(t, x)\>_\gm
-\<u, \pt (\eta_T^{2p'}\psi(t, x))\>_\gm \right)
\right|_{t=0}^T +\int_0^T\<u,  \Box_\gm (\eta_T^{2p'}\psi(t, x))\>_\gm dt\\
&=&
-\e\int_{\Omega}\left(\la_1 u_0(x) +u_1(x)\right)\phi_{\la_1}(x) dV_\gm
\\
&&+\int_0^T\int_{\Omega}u
\psi
\partial_t^2\eta_T^{2p'} dV_\gm dt+2\int_0^T\int_{\Omega}u(\partial_t\eta_T^{2p'})(\pt \psi) dV_\gm dt,
\end{eqnarray*}
which yields
\begin{equation}\label{thm1step5}
\begin{aligned}
&\e\int_{\Omega}\left(\la_1 u_0(x) +u_1(x)\right)\phi_{\la_1}(x) dV_\gm +\int_0^T\int_{\Omega}|u|^p \eta_T^{2p'} e^{-\la_1t}\phi_{\la_1}(x) dV_\gm dt\\
=&\int_0^T\int_{\Omega}ue^{-\la_1t}\phi_{\la_1}(x) (\partial_t^2\eta_T^{2p'}  -2\la_1 \partial_t\eta_T^{2p'}) dV_\gm dt \ .
\end{aligned}
\end{equation}
Noticing that
$$\partial_t^2\eta_T^{2p'}  -2\la_1 \partial_t\eta_T^{2p'}=\CO(T^{-1}(\eta_T^*)^{2p'-2})\ ,$$
As above, by combining H\"{o}lder's inequality and \eqref{phi0la},
the right hand side of  \eqref{thm1step5} is controlled by
\begin{equation}\label{IV}
\begin{aligned}
 & C T^{-1} \int_0^T\int_{\Omega} |u| (\eta_T^*)^{2p'-2}\phi_0^{1/p} \phi_0^{-1/p}e^{-\la_1t}\phi_{\la_1}(x) dV_\gm dt\\
\le &CT^{-1}
\| u (\eta_T^*)^{2p'-2}\phi_0^{1/p} \|_{L^p ([0,T]\times\Om)}
\| \phi_0^{-1/p}e^{-\la_1t}\phi_{\la_1} \|_{L^{p'} (\cup_{t\in [T/2,T]}D_1(t))}
\\
\le &CT^{-1+(2-\frac{p'}{2})\frac{1}{p'}}(\ln T)^{-\frac1p}\left(\int_0^T\int_{\Omega}|u|^p (\eta_T^*)^{2p'} \phi_0 dV_\gm dt\right)^{\frac 1p}\ .
\end{aligned}
\end{equation}

We then conclude from \eqref{thm1step5}, \eqref{IV} and \eqref{thm1step3}  that
\begin{equation}\label{thm1step6}
\e ^pT^{2-\frac p2}\ln T\les \int_0^T\int_{\Omega}|u|^p(\eta_T^*)^{2p'}\phi_0 dV_\gm dt
\les \ T^{3-2p'}\ln T\ ,
\end{equation}
which gives us the desired lifespan estimate  \eqref{lifespan1pc2}  for $1<p<p_c(2)$.

\subsection{
Critical case}

For the lifespan estimate of the critical power, we need one more lemma, which is similar to Lemma 4.2 in \cite{LT1}.
For completeness, we give a proof.
\begin{lem}\label{lembq}
Let $q>0$, $\la_{1}
\in (0,\la_{0})$ and
\[
b_q(t, x)=\int_0^{\la_1}e^{-\lambda t}\phi_{\lambda}(x)\lambda^{q-1}d\lambda\ ,
\]
then we have
\begin{enumerate}
  \item $b_{q}(t, x)$ satisfies following identities
$$\pt b_{q}(t, x)=-b_{q+1}(t, x),\ \ \pt^2 b_{q}(t, x)=b_{q+2}(t, x)=
\Delta_\gm b_q(t, x) \ .$$
  \item For
$x\in D_1(t)$, we have
\begin{equation}\label{bq3}
b_q(t, x)\les\left\{
\begin{array}{ll}
(t+R_1)^{-q}\ &\mbox{if}\ 0<q<\frac 12,\\
(t+R_1)^{-\frac 12}\left(t+R_1+1-\int_R^rK(s)ds\right)^{\frac{1}{2}-q}\ &\mbox{if}\ q>\frac 12.\\
\end{array}
\right.
\end{equation}
and
\begin{equation}\label{bqge3}
b_q(t, x)\gtrsim\frac{\phi_0}{\ln(t+R_1)}(t+R_1)^{-q},~~~\forall q>0.
\end{equation}
\end{enumerate}
\end{lem}
\begin{proof} The first part is trivial.
 Concerning \eqref{bqge3}, let $\de_1<\min(\de_{0}R_{1}/R_{2}, \la_1 R_1)$, such that $\frac{\de_1}{t+R_1}\leq \frac{\de_1}{R_1}\leq \la_1$,
 then for any $\la\le\frac{\de_1}{t+R_1}$, we have
 $r\le (t+ R_2)/\de_0<(t+ R_1)/\de_1\le\la^{-1}$
 for $x\in D_1(t)$, so that we could apply  \eqref{H1} and \eqref{asyphi0} to obtain
\begin{equation*}
\begin{aligned}
b_q(t, x)\gtrsim&\int_{\frac{\de_1}{2(t+R_1)}}^{\frac{\de_1}{ t+R_1 }}e^{-\lambda t}\frac{\phi_0(x)}{\ln(\lambda R)^{-1}}\lambda^{q-1}d\lambda\\
\gtrsim&\frac{\phi_0(x)}{\ln(t+R_1)}\int_{\frac{\de_1}{2(t+R_1)}}^{\frac{\de_1}{ t+R_1 }}e^{-\lambda(t+R_1)}\lambda^{q-1}d\lambda\\
\gtrsim&\frac{\phi_0(x)}{\ln(t+R_1)}(t+R_1)^{-q}\int_{\frac{\de_1}{2}}^{\de_1}e^{-\theta}\theta^{q-1}d\theta
\gtrsim&\frac{\phi_0(x)}{\ln(t+R_1)}(t+R_1)^{-q}\ .
\end{aligned}
\end{equation*}

Turning to the upper bound,
 \eqref{bq3}, we know from  \eqref{H1} that
\begin{equation}\label{bqlower2}
b_q(t, x)
\les \int_{0}^{\la_1}e^{-\lambda(t+R_1)}e^{\la \int_R^rK(s)ds}\langle\la r\rangle^{-\frac12}\lambda^{q-1}d\lambda \ .
\end{equation}
If $\int_R^rK(s)\le \frac{t+R_1}{2}$, we get
\begin{equation}\label{bqlower2a}
b_q(t, x)\les \int_{0}^{\la_1}e^{-\frac{\la(t+R_1)}{2}}\la^{q-1}d\la
\les \ (t+R_1)^{-q}\int_0^{\infty}e^{-\la}\la^{q-1}d\la
\les \ (t+R_1)^{-q}.
\end{equation}
For the case $\frac{t+R_1}{2}\le\int_R^rK(s)\le t+R_1$, we have $r\thicksim t+R_1$.
If $q<1/2$,
the estimate for \eqref{bqlower2} becomes
\begin{equation*}
b_q(t, x)
\les\int_{0}^{\la_1}\langle\la(t+R_1)\rangle^{-\frac12}\lambda^{q-1}d\lambda
\les \ (t+R_1)^{-q}
\int_{0}^{\infty} \langle s\rangle^{-\frac12}s^{q-1}d s\les \ (t+R_1)^{-q}\ .
\end{equation*}
It remains to consider the case $q> \frac12$
and
$r\thicksim t+R_1$, for which we have
\begin{eqnarray*}
b_q(t, x)&
\les&   (t+R_1)^{-\frac12} \int_{0}^{\la_1}e^{-\lambda[t+R_1+1-\int_R^rK(s)ds]}\lambda^{q-3/2}d\lambda \\
&\les&
(t+R_1)^{-\frac12}\left(t+R_1+1-\int_R^rK(s)ds\right)^{\frac12-q}
\int_{0}^{\infty}
e^{-\la}\la^{q-\frac32}d\la
\\
&\les&
(t+R_1)^{-\frac12}\left(t+R_1+1-\int_R^rK(s)ds\right)^{\frac12-q}
\ .
\end{eqnarray*}
This completes the proof of \eqref{bq3}.
\end{proof}

For $M\in [2, T]\subset [2, T(\ep))$, we set
\[
Y[|u|^pb_q](M)=\int_1^M\left(\int_0^T\int_{\Omega}|u|^p(t,x)b_q(t,x)(\eta_{\sigma}^*(t))^{2p'} dV_\gm dt\right)\sigma^{-1}d\sigma,
\]
with
$q=\frac12-\frac1p$.
Then from \eqref{bqge3} and \eqref{thm1step6}, we know that
\begin{equation}\label{cristep1}
\begin{aligned}
M\frac{dY}{dM}=&\int_0^T\int_{\Omega}|u|^pb_q(\eta_{M}^*)^{2p'} dV_\gm dt\\
\ge &\int_{M/2}^M\int_{\Omega}|u|^p
\frac{\phi_0}{\ln M}
M^{-q}
(\eta_{M}^*)^{2p'}
 dV_\gm dt\\
\gtrsim&\e^pM^{2-\frac p2}\ln M \frac{M^{\frac1p-\frac12}}{\ln M}=\e^pM^{\frac32-\frac p2+\frac1p}=\e^p,
\end{aligned}
\end{equation}
where we used the fact that $p=p_c(2)$. On the other hand, as in (4.10) in \cite{LT1}, we have
\begin{equation*}
\begin{aligned}
Y(M)
=&\int_1^M
\left(\int_0^T\int_{\Om}b_q|u|^p(\eta_\sigma^*(t))^{2p'}dV_\gm dt\right)\sigma^{-1}d\sigma\\
=&\int_0^T\int_{\Om}b_q|u|^p\int_1^M(\eta_\sigma^*(t))^{2p'}\sigma^{-1}d\sigma dV_\gm dt\\
=&\int_0^T\int_{\Om}b_q|u|^p\int_{\frac tM}^t(\eta^*(s))^{2p'}s^{-1}ds dV_\gm dt\\
= &\int_0^T\int_{\Om}b_q|u|^p\int_{
\max(\frac tM,
\frac 1 2)}^1 (\eta(s))^{2p'}s^{-1}ds dV_\gm dt\\
\le&\int_{0}^T\int_{\Om}b_q|u|^p(\eta (t/M))^{2p'}\int_{\frac 12}^1 s^{-1}ds dV_\gm dt
\le 2 \int_0^T\int_{\Om}\eta_M^{2p'}b_q|u|^pdV_\gm dt\ .
\end{aligned}
\end{equation*} As $(\partial_t^2-\Delta_\gm )b_q=0$, we get
\begin{equation}\label{cristep2}
\begin{aligned}
Y(M)\les& \int_0^T\<|u|^p, \eta_M^{2p'}b_q\>_\gm dt
=\int_0^T\< \partial_t^2u-\Delta_\gm u, b_q\eta_{M}^{2p'}\>_\gm dt\\
=&-\left(\e\int_{\Omega}u_1(x)b_q(0, x) dV_\gm +\e\int_{\Omega}u_0(x) b_{q+1}(0, x) dV_\gm \right)\\
&+\int_0^T\int_{\Omega}u (\partial_t^2
-\Delta_\gm)(b_q\eta_M^{2p'})  dV_\gm dt\\
\le& \int_0^T\int_{\Omega} u\left(2\partial_tb_q\partial_t\eta_M^{2p'}+b_q\partial_t^2\eta_M^{2p'}\right)  dV_\gm dt:= I_1+I_2.\\
\end{aligned}
\end{equation}

By \eqref{bq3} and \eqref{bqge3}, we know that
\begin{equation*}
\begin{aligned}
&\int_{\frac M2}^M\int_{D_1(t)}b_q^{-\frac{1}{p-1}}b_{q+1}^{\frac{p}{p-1}} dV_\gm dt\\
\les&
\int_{\frac M2}^M\int_{D_1(t)}
b_q^{-\frac{1}{p-1}}b_{q+1}^{\frac{p}{p-1}}rdrdt\\
\les&\int_{\frac M2}^M\int_{D_1(t)}\left(\frac{\ln(t+R_1)}{\ln  r/R}\right)^{\frac{1}{p-1}}\frac{(t+R_1)^{\frac{1}2-\frac1{p(p-1)}}}
{t+R_1+1-\int_R^rK(s)ds}drdt\\
\les&\int_{\frac M2}^M(\ln(t+R_1))^{\frac1{p-1}}(t+R_1)^{\frac{1}2-\frac1{p(p-1)}}\int_{\int_R^rK(s)ds\le\frac{t+R_1}{2}}r^{-1}(\ln
( r/R))^{-\frac{1}{p-1}}drdt\\
&+\int_{\frac M2}^{M}(t+R_1)^{\frac{1}2-\frac1{p(p-1)}}\int^{t+R_1}_{\frac{t+R_1}{2}}(t+R_1+1-
s)^{-1}dsdt\\
\les & M^{\frac{3p(p-1)-2}{2p(p-1)}}\ln M
=M^{\frac{p}{p-1}}\ln M
\ ,
\end{aligned}
\end{equation*}
then $I_1$, $I_2$ can be estimated as
\begin{equation*}
\begin{aligned}
I_1\les& \ M^{-1}\left(\int_0^T\int_{\Omega}|u|^pb_q (\eta_M^*)^{2p'} dV_\gm dt\right)^{\frac1p}
\left(\int_{\frac{M}{2}}^{M}\int_{D_1(t)}b_q^{-\frac{1}{p-1}}b_{q+1}^{\frac{p}{p-1}} dV_\gm dt\right)^{\frac{p-1}p}\\
\les&\ \left(\ln M\right)^{\frac{p-1}p}\left(\int_0^T\int_{\Omega}|u|^pb_q (\eta_M^*)^{2p'}  dV_\gm dt\right)^{\frac1p}\\
\end{aligned}
\end{equation*}
and
\begin{equation*}
\begin{aligned}
I_2\les&\  M^{-2}\left(\int_0^T\int_{\Omega}|u|^pb_q\eta_M^{*2p'} dV_\gm dt\right)^{\frac1p}
\left(\int_{\frac{M}{2}}^{M}\int_{D_1(t)}b_q dV_\gm dt\right)^{\frac{p-1}p}\\
\les&M^{-2}\left(\int_0^T\int_{\Omega}|u|^pb_q\eta_M^{*2p'} dV_\gm dt\right)^{\frac1p} \left(\int_{\frac M2}^M\int_{\int_R^rK(s)ds\le t+R_1}(t+R_1)^{-q }rdrdt\right)^{\frac{p-1}p}\\
\les&\left(\int_0^T\int_{\Omega}|u|^pb_q\eta_M^{*2p'} dV_\gm dt\right)^{\frac1p}.\\
\end{aligned}
\end{equation*}
In summary, recalling \eqref{cristep1},
we obtain
$$MY'(M)=
\int_0^T\int_{\Omega}|u|^pb_q(\eta_{M}^*)^{2p'} dV_\gm dt
\gtrsim \left(\log M\right)^{1-p}Y^p(M)\ ,
MY'(M)\gtrsim \ep^p\ ,
$$
for any $M\in [2, T]\subset [2, T(\ep))$.
Let $M=T$, we see that
\begin{equation}\label{Yp}
Y^p(T)\le C T\left(\log T\right)^{p-1}Y'(T),\
\ep^p\les \ TY'(T)\ ,
\end{equation}
for any $T\in [2, T(\ep))$.
With the help of
 \eqref{Yp},
the desired lifespan estimate
$$T(\ep)\le \exp(c\ep^{-p(p-1)})$$
 for $p=p_c(2)$ follows directly from
the following lemma with $p_1=p_2=p_c(2)$
and $\delta=\e^p$.
\begin{lem}
{\bf (Lemma 3.10 in \cite{ISWa})} Let $2<t_0<T$, $0\le \phi\in C^1([t_0, T))$. Assume that
\begin{equation}
\left\{
\begin{aligned}
& \delta\le K_1t\phi'(t), \quad t\in (t_0, T), \\
& \phi(t)^{p_1}\le K_2t(\log t)^{p_2-1}\phi'(t), \quad t\in (t_0, T)\\
\end{aligned}
\right.
\end{equation}
with $\delta, K_1, K_2>0$ and $p_1, p_2>1$. If $p_2<p_1+1$, then there exist positive constants $\delta_0$ and $K_3$(independent of $\delta$) such that
\begin{equation}
\begin{aligned}
T\le \exp\left(K_3\delta^{-\frac{p_1-1}{p_1-p_2+1}}\right)
\end{aligned}
\end{equation}
when $0<\delta<\delta_0$.
\end{lem}

\section{Test function $\Delta_{\gm_1+\gm_3} \Phi_\la=\la^2\Phi_\la$}\label{sec:AE3}
In this section, we consider the generalized eigenvalue problem of the Laplacian operator on $(\Om,\gm)$,
Theorem \ref{thm2}, when the metric $\gm$ is an exponential perturbation of the spherically symmetric, long range asymptotically Euclidean metric $\gm_1$.

\subsection{Spherically symmetric case}
Let us begin with the spherically symmetric case,
considering \eqref{eq-eigen} with
$\Om=\overline{B_{R}}^{c}$ and
$\gm=\gm_{1}$:
$$\Delta_{\gm_{1}} \phi_\la=\la^2\phi_\la, x\in  \Om;\ \phi_\la=0, x\in \pa \Om\ .
$$
If we look at the corresponding problem on $\R^2$
$$\Delta_{\gm_{1}} h_\la=\la^2 h_\la, x\in  \R^2;\ h_\la=1, r=\la^{-1}\ ,
$$ it is known from Liu-Wang \cite{LWp} that there exists a small positive constant $\lambda_2$ such that
$$h_\la(r)\simeq \<r\la\>^{-1/2}e^{\la\int^{r}_R K(s)ds}\ , \ \forall \lambda\in(0, \lambda_2)\ .$$
By comparing the desired solution with $h_\la$, we shall prove that there exists $\phi_{\la}$ such that
\beeq \phi_\la\le h_\la, \phi_\la\ge h_\la-h_\la(R),\eneq
\beeq\label{eq-key0}
\phi_\la\sim \left\{
\begin{array}{ll}
	\frac{\phi_0(r)}{\phi_0(\la^{-1})}\simeq \frac{\ln r/R}{\ln (R\la)^{-1}} & r\le \la^{-1},\\
	h_\la\simeq  \<r\la\>^{-1/2}e^{\la\int^{r}_R K(s)ds} & r\ge \la^{-1},
\end{array}
\right.
\eneq
for $\la\in (0, \min(1/(2R), \la_2))$.

\begin{lem}\label{lem5}
 Let $n=2$. Then for any $\lambda\in (0, \la_2)$, the following boundary value problem
\begin{equation}
\label{phiramda}
\left\{
\begin{aligned}
& \Delta_{\gm_1} \phi_{\lambda}(x)=\lambda^2\phi_{\lambda}(x),~~x\in \overline{B_R}^c,\\
&\phi_{\lambda}(x)\Big|_{\partial B_R}=0,\\
&\phi_{\lambda}(x)-h_\la(x)
\rightarrow 0
,~~~|x|\rightarrow \infty\\
\end{aligned}
\right.
\end{equation}
admits a spherically symmetric  solution satisfying
\begin{equation}\label{prophiramda}
\begin{aligned}
0\le h_\la(r)-h_\la(R)\le \phi_{\lambda}(x)\leq h_\la(r),~~x\in B_R^c\ .
\end{aligned}
\end{equation}
\end{lem}

\begin{proof}
Firstly we consider the following boundary value problem
\begin{equation}
\label{phiramda1}
\left\{
\begin{aligned}
& \Delta_{\gm_1} \widetilde{\phi}_{\lambda}(x)=\lambda^2\widetilde{\phi}_{\lambda}(x),~~x\in \overline{B_R}^c,\\
&\widetilde{\phi}_{\lambda}(x)\Big|_{\partial B_R}=-h_{\lambda}(x)\Big|_{\partial B_R},\\
&\widetilde{\phi}_{\lambda}(x)\to 0,~~~|x|\rightarrow \infty .\\
\end{aligned}
\right.
\end{equation}
By standard variational argument with functional $I[u]=\int_{\overline{B_R}^c} (\la^2 |u|^2+g_{1}^{jk}\nabla_j u \nabla_k u) dV_{\gm_{1}}$, we see that there
admits a unique solution $\widetilde{\phi}_{\lambda}(x)\in H^1(\overline{B_R}^c)\cap C^\infty(\overline{B_R}^c)$.
 Let
\[
\phi_{\lambda}(x)=\widetilde{\phi}_{\lambda}(x)+h_{\lambda}(x),
\]
then $\phi_{\lambda}(x)$ satisfies the boundary value problem \eqref{phiramda}.

It remains to prove
\[
\widetilde{\phi}_{\lambda}(x)\in [- h_\la(R), 0],~~~x\in \overline{B_R}^c\ .
\]
For any $C_0>0$, by  \eqref{phiramda1}, we know that
\begin{equation}\nonumber
\left\{
\begin{aligned}
& -\Delta_{\gm_1} \left(\widetilde{\phi}_{\lambda}(x)-C_0\right)+\lambda^2\left(\widetilde{\phi}_{\lambda}(x)-C_0\right)
=-\lambda^2C_0<0,~~x\in \overline{B_R}^c,\\
&\left(\widetilde{\phi}_{\lambda}(x)-C_0\right)\Big|_{\partial B_R}=\left(-h_{\lambda}(x)-C_0\right)\Big|_{\partial B_R}< 0,\\
&\widetilde{\phi}_{\lambda}(x)-C_0\to -C_0<0,~~~|x|\rightarrow \infty.\\
\end{aligned}
\right.
\end{equation}
By maximum principle, we have
\[
\widetilde{\phi}_{\lambda}(x)-C_0\leq 0,~~~x\in \overline{B_R}^c,
\]
for any $C_0>0$,
which means
\[
\widetilde{\phi}_{\lambda}(x)\leq 0,~~~x\in \overline{B_R}^c.
\]

Similarly, it is clear that
\begin{equation}\nonumber
\left\{
\begin{aligned}
& -\Delta_{\gm_1} \left(\widetilde{\phi}_{\lambda}(x)+h_\la(R)\right)+\lambda^2\left(\widetilde{\phi}_{\lambda}(x)+h_\la(R)\right)
=\lambda^2 h_\la(R)>0,~~x\in \overline{B_R}^c,\\
&\left(\widetilde{\phi}_{\lambda}(x)+h_\la(R)\right)\Big|_{\partial B_R}=\left(-h_{\lambda}(x)+h_\la(R)\right)\Big|_{\partial B_R}= 0,\\
&\widetilde{\phi}_{\lambda}(x)+h_\la(R)\to h_\la(R)>0,~~~|x|\rightarrow \infty.\\
\end{aligned}
\right.
\end{equation}
By maximum principle, we have
\[
\widetilde{\phi}_{\lambda}(x)+h_\la(R)\ge 0,~~~x\in \overline{B_R}^c,
\]
and this completes the proof of Lemma \ref{lem5}.
\end{proof}
With the help of Lemma \ref{lem5}, we get
\begin{equation}\label{eq-equiv1}\phi_{\lambda}(x)\simeq h_\la(r),
\eneq
for any $r\ge \la^{-1}$ with $\la< \min(\la_2, (2R)^{-1})$.

\begin{lem}
\label{philbd}
Let $\phi_{\lambda}(x)$ be the function in Lemma \ref{lem5} and $0<\la< \min(\la_2, (2R)^{-1})$. Then we have
\begin{equation}\label{lboundofvarphi}
\begin{aligned}
\phi_{\lambda}(x)\simeq  \frac{\ln r/R}{\ln 1/(\la R)}, \ \forall R\le r\le \la^{-1}.
\end{aligned}
\end{equation}
\end{lem}
\begin{proof}
Let $\de>0$ to be determined and
\[
F_{\lambda}(x)=\de\frac{\phi_0(x)}{\phi_0(\la^{-1})}h_{\lambda}(x)-\phi_{\lambda}(x),
\]
then, by \eqref{eq-gm1},
\begin{equation}\nonumber
\left\{
\begin{aligned}
& -\Delta_{\gm_1} F_\la+\lambda^2F_{\lambda}(x)=-2 \frac{\de}{K^2(r)\ln(\la^{-1})}\partial_r h_{\lambda}(x) \pa_r \phi_0<0,~~x\in \overline{B_R}^c,\\
&F_{\lambda}(x)= 0,~~~|x|=R,\\
&F_{\lambda}(x)= \de-\phi_\la(x)\le 0,~~~|x|=\la^{-1}\ ,\\
\end{aligned}
\right.
\end{equation}
when $\de>0$ is sufficiently small such that
$\phi_\la(x)\ge \de =\de h_\la(r)$ for $r=\la^{-1}$, in view of
\eqref{eq-equiv1}.
By the maximum principle, one obtain
\[
F_{\lambda}(x)\le 0,
\phi_{\lambda}(x)\ge \de\frac{\phi_0(x)}{\phi_0(\la^{-1})}h_{\lambda}(x)\ge \de h_\la(0)\frac{\phi_0(x)}{\phi_0(\la^{-1})}\ge \de'\frac{\ln r/R}{\ln (\la R)^{-1}} , \
\forall R\le r\le \la^{-1}\ ,
\]
for some $\de'>0$, which yields the desired lower bound.

On the other hand, let
\[
G_{\lambda}(x)=\phi_{\lambda}(x)-\frac{\phi_0}{\phi_0(\la^{-1})},
\]
then
\begin{equation}\nonumber
\left\{
\begin{aligned}
& -\Delta_{\gm_1} G_\la+\lambda^2G_{\lambda}(x)=-\la^2 \frac{\phi_0}{\phi_0(\la^{-1})}<0,~~x\in \overline{B_R}^c,\\
&G_{\lambda}(x)= 0,~~~|x|=R,\\
&G_{\lambda}(x)=\phi_\la(x)- 1\le 0,~~~|x|=\la^{-1}\ ,\\
\end{aligned}
\right.
\end{equation}
which gives us $G_{\lambda}(x)\le 0$.
\end{proof}

\subsection{Derivative estimates}
For future reference, we need to obtain more information concerning the behavior of $\phi_\la$.
At first,
we  claim that we have the following derivative estimates
\beeq\label{eq-dr-est}
\pa_{r}\phi_\la(r)\leq D_{0}\la^2 (r-R) \phi_\la (r)\ ,
|\pa_r^2\phi_\la(r)|\le D_0 \la^2\phi_\la (r)\ , \forall r\ge R
\eneq
for some constant $D_0$ independent of $\lambda\in (0, 1/(2R)]$.

As $\phi_\la$ is radial, by \eqref{dl3} and  \eqref{eq-gm1},  we have
$$\Delta_{\gm_1}\phi_\la=K^{-1}r^{-1}\pa_{r}( K^{-1} r \pa_{r}\phi_\la)\ ,$$
and so
\beeq
\label{s01}
\pa_{r}( K^{-1} r \pa_{r}\phi_\la)=\lambda^{2}Kr \phi_\la.
\eneq

As $\phi_\la$ is increasing, we get
$$
K^{-1} r \pa_{r}\phi_\la\le
\int_R^r\lambda^{2}\tau K\phi_\la(\tau) d\tau
\le \frac12\lambda^{2}(r^{2}-R^2)\|K\|_{L^\infty} \phi_\la(r)\ ,$$
that is, $
\pa_{r}\phi_{\lambda}
\le \|K\|^2_{L^\infty}\lambda^{2}(r-R)\phi_\la(r)$.

For the second order derivative of $\phi_{\lambda}$, by \eqref{s01}, we have
$$|\pa^{2}_{r}\phi_{\lambda}|
=|\lambda^{2}K^{2}\phi_\la-\big(\frac{ 1}{r}-\frac{K'}{K}\big)\pa_{r}\phi_\la|
\leq \lambda^{2}K^{2}\phi_{\lambda}+\frac{C}{r}\pa_{r}\phi_{\lambda} \ \les \ \lambda^{2}\phi_{\lambda}$$
for some $C>0$ due to \eqref{dl2}.

\subsection{Test function $\Delta_{\gm_1+\gm_3} \Phi_\la=\la^2\Phi_\la$}
Turning to the exponential perturbation of the spherically symmetric
asymptotically Euclidean metric, we consider the following problem
$$\Delta_{\gm_1+\gm_3} \Phi_\la=\la^2\Phi_\la, x\in  \Om;\ \Phi_\la=0, x\in \pa \Om\ .
$$
We would like to obtain similar bounds as that for $\gm_1$.
\begin{lem}
\label{asy1}There exists $\la_3>0$ such that for any $\la\in (0,\la_3)$, we can construct a solution $\Phi_\la$ such that
\beeq\label{eq-key}
\Phi_\la\sim \phi_\la \sim\left\{
\begin{array}{ll}
	\frac{\phi_0(r)}{\phi_0(\la^{-1})}  & r\le \la^{-1},\\
	h_\la\simeq  \<r\la\>^{-1/2}e^{\la\int^{r}_R K(s)ds} & r\ge \la^{-1}.
\end{array}
\right.
\eneq
\end{lem}

\subsection{Proof of Lemma \ref{asy1}}\label{sec:2}
As in \cite[Lemma 3.1]{LWp},
see also Wakasa-Yordanov \cite[Lemma 2.2]{WaYo19}, we introduce
 $\psi=\phi_\la-\Phi_\la$ where $\Delta_{\gm_{1}}\phi_\la=\lambda^{2}\phi_\la$. Then we are reduced to prove existence of $\psi$ and show smallness of $\|\psi\|_{L^{\infty}}$.

In fact,  $\psi$ satisfies
\begin{align}
\label{est-yx}
\Delta_{\gm}\psi=\Delta_{\gm}(\phi_\la-\Phi_\la)=\lambda^{2}\phi_\la-\lambda^{2}\Phi_\la-(\Delta_{\gm_{1}}-\Delta_{\gm})\phi_\la=\lambda^{2}\psi-W\phi_\la\ ,
\end{align}
and
$\psi |_{\partial B_R}=0$.
We claim that, for any $\la\in (0, \min(\al/2, 1/(2R), \la_2))$,
\beeq
\label{eq-Lq-forcing}
\|W\phi_\la\|_{ L^{q}}\les \la^2,\ \forall \ q\in [1,\infty]\ .
\eneq
Actually, we have
$$W\phi_\la=
(\gm^{jk}_{1}-\gm^{jk})
\pa_{j}\pa_{k}\phi_\la+
[g_1^{-1/2}\pa_j (g_1^{1/2}\gm_1^{jk})
-g^{-1/2}\pa_j (g^{1/2}\gm^{jk})]
 \pa_k\phi_\la\ .
$$
By
\eqref{eq-dr-est}
 and \eqref{dlfjia}
 with $\lambda<\min(\al/2, 1/(2R), \la_2)$, it is easy to see that
$$|W \phi_\la|\les \sum_{j,k}
\sum_{1\le |\be|\leq 2}
|\nabla^{\leq 1}g^{jk}_{3}||\nabla^{\be}\phi_\la|\les
\la^2\< r\> e^{(\la-\al)\int^{r}_{R}K(\tau)d\tau}
\les
\la^2 \< r\>e^{-\frac{\al}{2}\de_0 r}
\ ,$$
which gives us \eqref{eq-Lq-forcing}.

Standard elliptic theory ensures that there exists a unique weak $H^1_0$ solution $\psi$
to \eqref{est-yx}, which satisfies
$\psi\in H^3\cap C^\infty$.
  To show the smallness of $\| \psi\|_{L^{\infty}}$, we first take the (natural) inner product of \eqref{est-yx} with $\psi$ to get
$$-\langle \Delta_{\gm}\psi, \psi\rangle_\gm+\lambda^{2}\langle \psi, \psi\rangle_\gm=-\langle W\phi_\la,  \psi\rangle_\gm\ .$$
Thus by the Cauchy-Schwarz inequality and uniform elliptic condition \eqref{unelp} we have
$$\delta_{0}\|\nabla \psi\|^2_{L^{2}_{\gm}}+\lambda^{2}\|\psi\|^2_{L^{2}_{\gm}}\ \les \ \lambda^{2}\|\psi\|_{L^{2}_{\gm}},$$
which yields
$$\|\psi\|_{L^{2}_x}\ \les \ 1, \ \|\nabla \psi\|_{L^{2}_{x}}\ \les \ \lambda\ .$$
In view of the equation, we obtain
$$\|\Delta_\gm \psi\|_{L^{2}_x}\ \les \ \la^{2} \ .$$
As $\psi\in H^2\cap H^1_0$, we get from \cite[Theorem 9.11-13]{GT01} that
$$\|\nabla^2\psi\|_{L^{2}_x}\ \les \ \|\psi\|_{H^1}+ \|\Delta_\gm \psi\|_{L^{2}_x}\ \les 1\ \ .$$

Applying the Gagliardo-Nirenberg inequality (for the exterior domain, see, e.g., \cite{CM}), we get
\beeq\label{eq-L4}\|\psi\|_{L^{4}_x}\les\
\|\psi\|_{L^{2}_x}^{1/2}
\|\nabla \psi\|_{L^{2}_x}^{1/2}
\les\
\la^{1/2}\ ,\
\|\nabla \psi\|_{L^{4}_x}\les\
\|\nabla \psi\|_{L^{2}_x}^{1/2}
\|\nabla^2 \psi\|_{L^{2}_x}^{1/2}
\les\
\la^{1/2}\ .
\eneq
By Sobolev embedding, we obtain
\beeq\label{eq-Linfty}\|\psi\|_{L^{\infty}_x}\les
\| \psi\|_{L^{4}_x}+\|\nabla \psi\|_{L^{4}_x}
\les\
\la^{1/2}\ .
\eneq

Notice that
$\phi_\la$ tends to zero as $r\to R$,
the estimate \eqref{eq-Linfty} is not good  enough to be controlled by $\phi_\la$ near the boundary. For that purpose, we observe that, in view of the equations,  \eqref{eq-Lq-forcing} and \eqref{eq-L4}, we have
$$\|\Delta_\gm \psi\|_{L^{4}_x}\les \ \la^{2}\ .$$
As $\psi\in W^{2,4}\cap W^{1,4}_0$, we get from \cite[Theorem 9.11-13]{GT01} that
$$\|\nabla^2 \psi\|_{L^{4}_x}\les \|\psi\|_{W^{1,4}}+ \|\Delta_\gm \psi\|_{L^{4}_x}\les \ \la^{1/2}\ .$$
Another application of the Sobolev embedding gives us
\beeq\label{eq-Linfty2}\|\nabla \psi\|_{L^{\infty}_x}\les
\|  \nabla\psi\|_{L^{4}_x}
+\|\nabla\nabla\psi\|_{L^{4}_x}
\les\
\la^{1/2}\ ,
\eneq
which, together with the boundary condition
$\psi=0$ for $r=R$, yields
\beeq\label{eq-Linfty3}\|\psi(r,\theta)\|_{L^{\infty}_\theta}\les\
\la^{1/2}(r-R)\ .
\eneq

In summary, by
\eqref{eq-Linfty}, \eqref{eq-Linfty3} and
\eqref{eq-key0}
we have obtained
\beeq\label{eq-Linfty4} |\psi(r,\theta)|\les\
\la^{1/2}\min\{|x|-R, R\}\les\
\la^{1/2}\phi_0(r)\ll \phi_\la
\ ,
\eneq
which gives us the desired estimate
\eqref{eq-key} for $r\ge R$ and $\la<\la_3\ll \min(1, \la_2, \al/2, 1/(2R))$.

\section{
General obstacles
}\label{sec:diffeo}
In this section, we present the proof of Lemma \ref{thm-diffeo}, which enables us to reduce the problems with general obstacles to the problem exterior to a disk.

\subsection{Star-shaped obstacles}
It turns out that we could construct an explicit diffeomorphism, when
the obstacle $\mathcal{K}$ (the interior of the Jordan curve  $\pa D$) is star-shaped.
As $\mathcal{K}$
is star-shaped, there exists a smooth $R:\mathbb{S}^{1}\to \R_{+}$:
$$\mathcal{K}=\{(r\cos\theta, r\sin\theta), r< R(\theta)\in (0,\infty)\}\ .$$
  At first, we set $\de_2>0$ such that
$$R(\theta)\in [2\de_2, \de_2^{-1}/2]\ ,$$
$R_3=\de_2$ and $R_4=\de_2^{-1}$.
Let $\mu$ be a decreasing cut-off function such that
$\mu=0$ for $r>2$ and $\mu=1$ for $r<1$.
We set
\beeq
 f(r,\theta)=\mu(\frac r{R_4})\frac{R_3}{R(\theta)}r+(1-\mu(\frac r{R_4}))r
\eneq
and introduce
\beeq
\mathcal{A}: (r,\theta)\to (f(r,\theta),\theta).
\eneq
Notice that
$f(R(\theta), \theta)=R_3$,
$f(r, \theta)=r$ for any $r>2R_4$, and
 $$\pa_r f(r,\theta)=\mu(\frac r{R_4})\frac{R_3}{R(\theta)}+(1-\mu(\frac r{R_4}))-
r \mu'(\frac r{R_4})\frac{R(\theta)-R_3}{R_4 R(\theta)}
\in \left[\frac{R_3}{R(\theta)}, 1+\frac{(2R_4)C }{R_4 }\right]
$$
and so
\beeq\label{diffeo}\pa_r f(r,\theta)\in [2\de_2^2, 1+2C], \ \forall r, \theta\ , \eneq
where we have denoted $C=\|\mu'\|_{L^\infty}$ and used the fact that $\mu'=0$ unless $r\in [R_4, 2R_4]$.

It follows then that the map $$\mathcal{A}: D\to \Omega:=\R^2\backslash \overline{B_{R_3}}$$ is a smooth diffeomorphism mapping preserving the boundary, which is an identity map for $r>2 R_4$.

\subsection{General obstacles}
Let $\gamma$ be a Jordan curve in $\R^2$ and $D$ be the unbounded component of the two components divided by  $\gamma$. Without loss of generality, we assume $0\in \R^2\setminus \bar{D}$. Let
$$\hat{D}=\{z; z=x+iy, (x,y)\in D\}\subset \mathbb{C} \ , \check{D}=\{z;\frac{1}{z}\in \hat{D}\}.$$
Then by the Riemann mapping theorem (see, e.g., \cite[Chapter 5, Theorem 4.1]{T11}), there exist $R>0$ and (holomorphic) diffeomorphism $\mathcal{B}$
preserving the boundary such that
$$\mathcal{B}: \overline{\check{D}\cup\{0\}} \to \overline{B_{1/R}}, \ \mathcal{B}(0)=0, \ \mathcal{B}'(0)=1\ ,$$
which gives us a diffeomorphism mapping
 $$\mathcal{B}: \check{D} \to B_{1/R}\backslash\{0\}\ .$$
 Let $\delta\in (0,1)$ be a small parameter and $\mu(z)$ be a smooth cut-off function such that $\mu=0$ when $|z|\leq 2$ and $\mu=1$ when $|z|\geq 3$. Then we define the map
\beeq
\label{morp1}
\mathcal{H}(z)=\mu_{\delta}(z)\mathcal{B}(z)+(1-\mu_{\delta}(z))z,\ \mu_{\delta}(z)=\mu(\frac{z}{\delta})\ .
\eneq
We claim that there exists a small $\delta_{1}>0$ such that
\beeq
\label{morp2}
\mathcal{H}(z): \overline{\check{D}} \to \overline{B_{1/R}\backslash\{0\}}\eneq
 is a diffeomorphism mapping when $\delta<\delta_1$. Thus it is easy to see that
 $$\mathcal{A}(x,y)=\frac{1}{\mathcal{H}(\frac{1}{x+iy})}: D\to \Omega$$
is a desired diffeomorphism mapping. Hence we are reduced to showing the claim \eqref{morp2}.

By definition, we have $\mathcal{H}=\mathcal{B}$ when $|z|\geq 3\delta$ and $\mathcal{H}(z)=z$ when $|z|\leq 2\delta$. So we need only to consider the case $2\delta<|z|< 3\delta$. 
Let $\mathcal{B}:(x, y)\to (\tilde{F}, \tilde{G})$, then we have
\beeq
\label{morp6}
\mathcal{J}(\tilde{F},\tilde{G})(0)=\begin{pmatrix}
    \tilde{F}_x(0)  &   \tilde{F}_{y}(0) \\
   \tilde{G}_x(0)   &  \tilde{G}_y(0)
   \end{pmatrix}=\begin{pmatrix}
    1  &   0 \\
   0   &  1
   \end{pmatrix}\ .\eneq
   Let $\mathcal{H}: (x, y)\to (F, G)$, by definition \eqref{morp1} we have
\beeq \nonumber
F=x+\mu_{\delta}(\tilde{F}-x),\ G=y+\mu_{\delta}(\tilde{G}-y)
\eneq
and by \eqref{morp6}, for $r=\sqrt{x^2+y^2}<4\de$ small enough, we get that
$$F_x=1+\mathcal{O}(\delta), F_y=\mathcal{O}(\de),
G_x=\mathcal{O}(\delta), G_y=1+\mathcal{O}(\de)\ .$$
Then there exists $\de_3>0$, such that for any $\de\le \de_3$, we have
$\mathcal{H}(B_{3\de})\subset  B_{1/R}$ and
$$\mathcal{J}(F, G)(x,y)\in \left(\frac{1}{2}, \frac 32\right), \   \forall r< 4\de_3\ .$$
By inverse function theorem and compactness of $\bar B_{3\de_3}$,
there exists a uniform $\ep_1>0$ such that
$\mathcal{H}$ is a local diffeomorphism on $B_{\ep_1}(r_{0}, \theta_{0})$
for any
$r_{0}\in [0, 3\de_3]$ and
$\theta_{0}\in [0, 2\pi]$. Thus we are reduced to proving bijection of $\mathcal{H}$.

We firstly prove the injection.
For that purpose, we replace $(x, y)$ by polar coordinates $(r, \theta)$ and $(F, G)$ by $(R, \varphi)$. Recall that
$x\pa_y-y\pa_x=\pa_\theta$ and $x\pa_x+y\pa_y=r\pa_r$, we get
\beeq
\label{mor4}\varphi_\theta=1+\mathcal{O}(\de), R_{\theta}=\mathcal{O}(\de^2), \varphi_r=\mathcal{O}(1), R_r=1+\mathcal{O}(\de)\ , r<4\de\ .\eneq
Let $\Gamma_r=\mathcal{H}(\pa B_r)$ ($r>0$) be closed curves.
With possibly shrinking $\de_3>0$, we could assume, for some $C>0$,
\beeq
\label{mor7}
|\mathcal{O}(t)|\leq C t, C\de\le \frac{1}{4}\ .
\eneq
 Then by \eqref{mor4} we have $\varphi_\theta\in [3/4, 5/4]$ and thus the winding number of
 $\Ga_r$ about the origin $0$ must be $1$, which shows that $\Gamma_r$ is diffeomorphic to the unit circle $\mathbb{S}^1$.
To complete the proof of injection, we need only to show $\{\Ga_r\}_{0<r\le 3\de}$ are disjoint, for which it suffices to prove
$$\Gamma_{r_1}\cap \Gamma_{r_2}=\emptyset\ , \ 0<|r_1-r_2|<\frac{\de\ep}{2}<\frac{\ep}{2}\ .$$
In fact, if $\Gamma_{r_1}\cap \Gamma_{r_2}\neq\emptyset$ for some $r_1\neq r_2$ then $\varphi(r_1, \theta_1)=\varphi(r_2, \theta_2)$ for some $\theta_1, \theta_2$.
Without loss of generality, we assume $\theta_2>\theta_1$.
With $\theta_t=(1-t)\theta_0+t\theta_1$,
$r_t=(1-t)r_0+tr_1$,
by \eqref{mor4} and \eqref{mor7}, we have
$$0=\frac{d}{dt}\int_1^2 \varphi(r_t, \theta_t)dt=\int_1^2 \varphi_r \frac{d}{dt}r_t+\varphi_\theta \frac{d}{dt}\theta_t dt\geq
(1-C\de)(\theta_2-\theta_1)-C(r_2-r_1) ,$$
which yields $\theta_2-\theta_1\leq \frac{\ep\de}{2}\frac{C}{1-C\delta}<\frac{\ep}{2}$, that is
$$(r_2, \theta_2)\in B_{\ep}(r_1, \theta_1).$$
This is a contradiction to the  diffeomorphism  of $\mathcal{H}$ on $B_{\ep}(r_1, \theta_1).$

Let  $\Gamma^{\circ}_{r}$ denote the inner points of $\Gamma_{r}$, as $\Gamma_{0}=\{0\}$ and $\Gamma_{3\de}=\mathcal{B}(\pa B_{3\de})$, we see that $\Gamma_r$ is increasing about $r$ ($\Gamma^{\circ}_{r_1}\subset \Gamma^{\circ}_{r_2}$, if $r_1<r_2$).
As $\Ga_r$ depends continuously on $r$, then it is clear  that we have
$\mathcal{H}(B_{3\de})=\Gamma^{\circ}_{3\de}=\mathcal{B}(B_{3\de})$.
This completes the proof of \eqref{morp2}.

\begin{center}
ACKNOWLEDGMENTS
\end{center}
The first author is supported by NSF of Zhejiang Province(LY18A010008) and NSFC 11771194.
The second and fourth author were supported in part by NSFC 11971428.
The third author has been partially supported by
Grant-in-Aid for Scientific Research (B) (No.18H01132) and Grant-in-Aid for Scientific Research (B) (No. 19H01795) and Young Scientists Research (No. 20K14351),
Japan Society for the Promotion of Science.


\vskip10pt



\begin{thebibliography}{10}


\bibitem{CG}
D. Catania, V. Georgiev, Blow-up for the semilinear wave equation in the Schwarzschild metric, Differential Integral
Equations, 19 (2006) 799-830.

\bibitem{CM}
F. Crispo, P. Maremonti, An Interpolation Inequality in Exterior Domains, REND. SEM. MAT. UNIV. PADOVA, 112 (2004).

\bibitem{Du}
Y. Du, J. Metcalfe, C. D. Sogge, Y. Zhou, Concerning the Strauss conjecture and almost global existence for nonlinear Dirichlet-wave equations in 4-dimensions, Comm. Partial Differential Equations 33(7-9)(2008) 1487-1506.
\bibitem{Du1}
Y. Du, Y. Zhou, The lifespan for nonlinear wave equation outside of star-shaped obstacle in
three space dimensions, Comm. Partial Differential Equations 33 (7-9)(2008) 1455-1486.

\bibitem{Georgive}
V. Georgiev, H. Lindblad, C. D. Sogge, Weighted Strichartz estimates and global existence for semilinear wave equations,
Amer. J. Math. 119(6) (1997) 1291-1319.



\bibitem{GT01}
D. Gilbarg, N. S. Trudinger,
\newblock Elliptic partial differential equations of second order,
\newblock Classics in Mathematics. Springer-Verlag, Berlin, 2001.
\newblock Reprint of the 1998 edition.

\bibitem{Glassey1}
R.T. Glassey, Finite-time blow-up for solutions of nonlinear wave equations, Math. Z. 177 (1981) 323-340.
\bibitem{Glassey2}
R.T. Glassey, Existence in the large for $\Box u=F(u)$ in two space dimensions, Math. Z. 178 (1981) 233-261.




\bibitem{Hidano}
K. Hidano, J. Metcalfe, H. F. Smith, C. D. Sogge, Y. Zhou, On abstract Strichartz estimates and the
Strauss conjecture for nontrapping obstacles, Trans. Amer. Math. Soc. 362(5) (2010) 2789-2809.


\bibitem{ISWa}
M. Ikeda, M. Sobajima, K. Wakasa,
Blow-up phenomena of semilinear wave equations and their weakly coupled systems, J. Differential Equations, 267(9) (2019) 5165-5201.


\bibitem{Takamura6}
T. Imai, M. Kato, H. Takamura, K. Wakasa, 
The sharp lower bound of the lifespan of solutions to semilinear
  wave equations with low powers in two space dimensions.
In {\em {Asymptotic analysis for nonlinear dispersive and wave
  equations. Proceedings of the international conference on asymptotic analysis
  for nonlinear dispersive and wave equations, Osaka University, Osaka, Japan,
  September 6--9, 2014}}, pages 31--53. Tokyo: Mathematical Society of Japan,
  2019.




\bibitem{John}
F. John, Blow-up of solutions of nonlinear wave equations in three space dimensions,
Manuscripta Math. 28 (1-3)(1979) 235-268.


\bibitem{LT1}
N. A. Lai, Z. H. Tu, Strauss exponent for semilinear wave equations with scattering space dependent damping, J. Math. Anal. Appl. 489 (2020) 124189.

\bibitem{Lai}
N. A. Lai, Y. Zhou, An elementary proof of Strauss conjecture, J. Funct. Anal. 267 (5)(2014) 1364-1381.

\bibitem{Lai2}
N. A. Lai, Y. Zhou, Finite time blow up to critical semilinear wave equation
outside the ball in 3-D, Nonlinear Anal. 125 (2015) 550-560.

\bibitem{Lai3}
N. A. Lai, Y. Zhou, Nonexistence of global solutions to critical semilinear wave equations in exterior domain in high dimensions, Nonlinear Anal. 143 (2016), 89-104.

\bibitem{Lai1}
N. A. Lai, Y. Zhou, Blow up for initial boundary value problem of critical
semilinear wave equation in 2-D,  Commun. Pure Appl. Anal. 17 (4)(2018) 1499-1510.


\bibitem{LZ}
N. A. Lai, Y. Zhou, Global existence of solutions of the critical semilinear wave equations with variable coefficients outside obstacles, Sci. China Math. 54(2) (2011) 205-220.





\bibitem{Li}X. Li, G. Wang, Blow up of solutions to nonlinear wave equations in 2D exterior domains, Arch. Math. 98 (2012) 265-275.

\bibitem{LLM}
Y. H. Lin, N. A. Lai, S. Ming, Lifespan estimate for semilinear wave equation in Schwarzschild spacetime, Appl. Math. Lett. 99 (2020) 105997, 4 pp.


\bibitem{Lindblad3}
H. Lindblad, Blow-up for solutions of $\Box u=|u|^p$ with small initial data, Comm. Partial Differential Equations 15 (6) (1990)
757-821.

\bibitem{LMS}
H. Lindblad, J. Metcalfe, C. D. Sogge, M. Tohaneanu, C. B. Wang, The Strauss conjecture on Kerr black hole backgrounds, Math. Ann., 359(3-4) (2014) 637-661.



\bibitem{Lindblad2}H. Lindblad, C. D. Sogge, Long-time existence for small amplitude semilinear wave equations,
Amer. J. Math. 118 (5) (1996) 1047-1135.

\bibitem{LWp}
M. Y. Liu, C. B. Wang,
\newblock
The blow up of solutions to semilinear wave equations on asymptotically Euclidean manifolds, Preprint. \arXiv{1912.02540}, 2019.

\bibitem{MMT}
J. Marzuola, J. Metcalfe, D. Tataru, M. Tohaneanu,
Strichartz estimates on Schwarzschild black hole backgrounds, Comm. Math. Phys. (293) 2010 37-83.

\bibitem{Metcalfe}
J. Metcalfe, C. D. Sogge, Global existence for high dimensional quasilinear wave equations
exterior to star-shaped obstacles, Discrete Contin. Dyn. Syst. 28(4) (2010) 1589-1601.


\bibitem{MW}
J. Metcalfe, C. B. Wang, The Strauss conjecture on asymptotically flat space-times, SIAM J Math. Anal. 49 (2017) 4579-4594.

\bibitem{Schaeffer}
J. Schaeffer, The equation $\Box u=|u|^p$ for the critical value of $p$, Proc. Roy. Soc. Edinburgh
Sect. A 101 (1-2)(1985) 31-44.
\bibitem{Sideris}T.C. Sideris, Nonexistence of global solutions to semilinear wave equations in high dimensions, J.Differential Equations 52 (1984) 378-406.


\bibitem{SS}
 H. F. Smith, C. D. Sogge, On the critical semilinear wave equation outside convex obstacles, J. Amer. Math. Soc.
  8 (4)(1995) 879-916.

\bibitem{Smith}
H. F. Smith, C. D. Sogge, C. B. Wang, Strichartz estimates for Dirichlet-wave equations in two dimensions with applications, Trans. Amer. Math. Soc. 364 (2012) 3329-3347.

\bibitem{SW}
M. Sobajima, K. Wakasa, Finite time blowup of solutions to semilinear wave equation in an exterior domain, J. Math. Anal. Appl. 484 (2020), 123667.

\bibitem{SWa}
C. D. Sogge, C. B. Wang, Concerning the wave equation on asymptotically Euclidean manifolds, J Anal Math 112 (2010) 1-32.

\bibitem{Strauss}
  W. A. Strauss, Nonlinear scattering theory at low energy, J. Funct. Anal. 41 (1981) 110-133.


\bibitem{Takamura2}
  H. Takamura, Improved Kato's lemma on ordinary differential inequality and its
application to semilinear wave equations, Nonlinear Analysis 125 (2015) 227-240.

\bibitem{Takamura}
  H. Takamura, K. Wakasa, The sharp upper bound of the lifespan of solutions to critical semilinear wave equations in high
dimensions, J. Differential Equations 251 (2011) 1157-1171.


\bibitem{T11}
M. E. Taylor, \newblock  Partial differential equations I-III. Second edition. \newblock Applied Mathematical Sciences, 115-117. Springer, New York, 2011.




\bibitem{WaYo19}
K. Wakasa, B. Yordanov,
\newblock Blow-up of solutions to critical semilinear wave equations with variable coefficients, J. Differential Equations 266(9) (2019) 5360-5376.





\bibitem{WY1}
C. B. Wang, X. Yu, Concerning the Strauss conjecture on asymptotically Euclidean manifolds, J. Math. Anal. Appl. 379 (2011)
549-566.


\bibitem{Wang1}
C. B. Wang, Recent progress on the Strauss conjecture and related problems (in Chinese), Sci. Sin. Math. 48 (2018)
111-130.

\bibitem{Wang2}
C. B. Wang, Long time existence for semilinear wave equations on asymptotically flat space-times, Comm. Partial Differential Equations 42 (7) (2017) 1150-1174.

\bibitem{YorZh06}
B. Yordanov, Q. S. Zhang,
Finite time blow up for critical wave equations in high dimensions, J. Funct. Anal. 231(2) (2006) 361-374.

\bibitem{Yu}
X. Yu, Generalized Strichartz estimates on perturbed wave equation and applications on Strauss conjecture, Differential and Integral Equations 24(5-6) (2011) 443-468.

\bibitem{Zha}
D. B. Zha, Y. Zhou, Lifespan of classical solutions to quasilinear wave equations outside of a star-shaped obstacle in four space dimensions, J. Math. Pures Appl. 103 (2015) 788-808.


\bibitem{Zhou2}
Y. Zhou, Blow up of solutions to semilinear wave equations with critical exponent in high
dimensions, Chinese Ann. Math. Ser. B 28(2)(2007) 205-212.
\bibitem{Zhou3}
Y. Zhou, Lifespan of classical solutions to $\Box u=|u|^p$ in two space dimensions, Chin. Ann. Math. Ser.B 14 (1993) 225-236.
\bibitem{Zhou5}
Y. Zhou, Blow up of classical solutions to $\Box u=|u|^{1+\alpha}$ in three space dimensions, J. Partial Differential Equations 5 (1992)
21-32.

\bibitem{Zhou7} Y. Zhou, W. Han, Blow-up of solutions to semilinear wave equations with variable coefficients and boundary, J. Math. Anal. Appl. 374 (2011) 585-601.




\end{thebibliography}
\end{document}